\documentclass[a4paper,pdftex]{amsart}

\usepackage{geometry}
\title{ON INEQUALITIES BETWEEN UNKNOTTING NUMBERS AND CROSSING NUMBERS OF SPATIAL EMBEDDINGS OF TRIVIALIZABLE GRAPHS AND HANDLEBODY-KNOTS}
\author{YUTA AKIMOTO}

\usepackage{amsmath} 
\usepackage{amsfonts}
\usepackage{mathrsfs}
\usepackage{amssymb}
\usepackage{amsthm}
\usepackage{bm}
\usepackage[new]{old-arrows}
\usepackage{wrapfig}
\usepackage{graphicx}
\usepackage{color}
\usepackage[labelsep=colon]{caption}
\usepackage{multicol}
\usepackage{units}
\usepackage[all]{xy}
\usepackage{tikz}
\usepackage{here}
\usepackage[subrefformat=parens]{subcaption}
\newtheorem{thm}{Theorem}[section]
\newtheorem{prop}[thm]{Proposition}
\newtheorem{lemma}[thm]{Lemma}

\theoremstyle{definition}

\newtheorem{remark}[thm]{Remark}

\begin{document}
\maketitle

\begin{abstract}
We study relations between unknotting number and crossing number of a spatial embedding of a handcuff-graph and a theta curve. It is well known that for any non-trivial knot $K$ twice the unknotting number of $K$ is less than or equal to the crossing number of $K$ minus one. We show that this is extended to handlebody-knots. We also characterize the handlebody-knots which satisfy the equality. 
\end{abstract} 

\section{INTRODUCTION}
Let $L$ be a link in the 3-dimensional Euclidean space $\mathbb{R}^3$. The \textit{unknotting number} $u(L)$ is the minimal number of crossing changes\ (Fig.\ \ref{cc}) from $L$ to a trivial link. The \textit{crossing number} $c(L)$ is the minimal number of crossing points among all regular diagrams of $L$. It is well-known that $u(L)$ is less than or equal to half of $c(L)$\ (see for example\ \cite{unbound}). In\ \cite{unbound} Taniyama characterized the links which satisfy the equality as follows. 

\begin{figure}[h]
\centering
\includegraphics[height=0.12\linewidth]{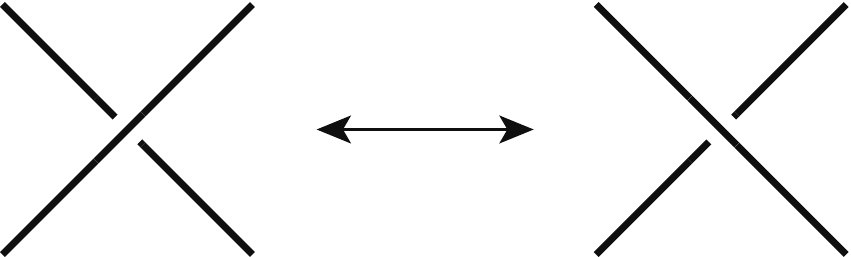}
\caption{}
\label{cc}
\end{figure}

\begin{thm} \cite[Theorem $1.5\,(2)$] {unbound}
Let L be a $\mu-$component link that satisfies the equality $u(L) =\dfrac{c(L)}{2}$. Then $L$ has a diagram $D = \gamma_1 \cup \cdots \cup \gamma_{\mu}$ such that each $\gamma_i$ is a simple closed curve on $\mathbb{R}^2$ and for each pair $i,\ j$, the subdiagram $\gamma_i \cup \gamma_j$ is an alternating diagram or a diagram without crossings.
\end{thm}

In\ \cite{uncr} Taniyama and the author showed that this inequality is not extended to spatial embeddings of planar graphs and this inequality is extended to spatial embeddings of trivializable planar graphs. Namely for any spatial embedding $f$ of a trivializable planar graph, $u(f)$ is less than or equal to half of $c(f)$. For example, a handcuff-graph and a theta-curve as illustrated in Fig.\ \ref{hanthe} are trivializable. We characterize the spatial embeddings of a handcuff-graph or a theta curve which satisfy the equality as follows.

\begin{figure}[H]
      \begin{tabular}{c}
      \begin{minipage}[b]{0.5\hsize}
        \begin{center}
          \includegraphics[height=0.3\linewidth]{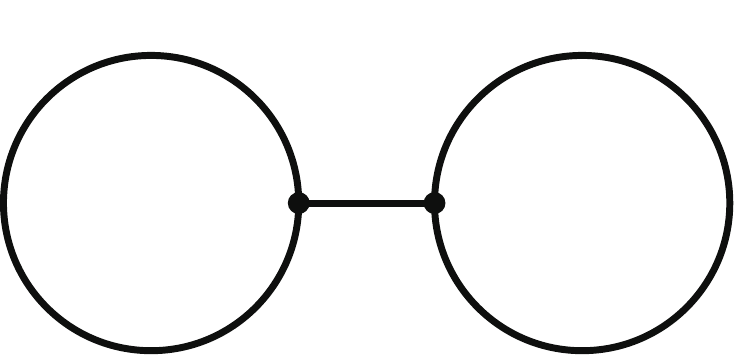}\\
          \ \\
          handcuff-graph
        \end{center}
      \end{minipage}
      \begin{minipage}[b]{0.5\hsize}
        \begin{center}
          \includegraphics[height=0.3\linewidth]{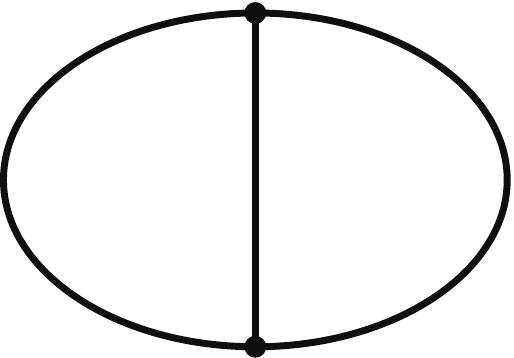}\\
          \ \\
          theta curve
        \end{center}
      \end{minipage}
      \end{tabular}
      \caption{}
      \label{hanthe}
\end{figure}

\begin{thm}\label{handcuff}
Let $G$ be a handcuff-graph and let $f$ be a spatial embedding of $G$. Then $f$ satisfies the equality $u(f)=\dfrac{c(f)}{2}$ if and only if $f$ has a diagram $D$ with the following conditions {\rm:} \\
$(1)\,$Each edge of $D$ has no self-crossings. \\
$(2)\,$All crossings of $D$ are crossings between two loops.\\
$(3)\,$Two loops of $D$ form an alternating diagram or a diagram without crossings.
\end{thm}

\begin{thm}\label{theta}
Let $G$ be a theta curve and let $f$ be a spatial embedding of $G$. Then $f$ satisfies the equality $u(f)=\dfrac{c(f)}{2}$ if and only if $f$ is trivial.
\end{thm}

We note that unknotting numbers of spatial embeddings of a theta curve is studied in \cite{ob}.

A \textit{handlebody-knot} is an embedded handlebody in the 3-dimensional Euclidean space $\mathbb{R}^3$, which is introduced by Ishii in \cite{handle}. Two handlebody-knots $H_1$ and $H_2$ are equivalent if there is an
orientation-preserving homeomorphism $h$ of $\mathbb{R}^3$ with $h(H_1) = H_2$. A \textit{spine} of a handlebody-knot $H$ is a spatial graph whose regular neightborhood is $H$. In this paper, we assume that spines have no degree 1 verticies. Any handlebody-knot $H$ can be represented by a spatial trivalent graph that is a spine of $H$. In particular, genus 2 handlebody-knot can be represented by a spatial embedding of a handcuff-graph or a theta curve. A \textit{crossing change} of a handlebody-knot $H$ is that of a spatial trivalent graph representing $H$. In\ \cite{unqc} Iwaliri showed that a crossing change of a handlebody-knot is an unknotting operation and give lower bounds of the unknotting numbers for handlebody-knots by the numbers of some finite Alexander quandle colorings. 

We have the following well-known relation between unknotting number and crossing number of classical knots.
\begin{prop} \label{12c1k}
Let $K$ be a nontrivial knot. Then $u(K) \leq \dfrac{c(K)-1}{2}$.
\end{prop}
In \cite{unbound} Taniyama characterized the knots which satisfy the equality as follows.
\begin{thm} \cite[Theorem1.4\,(2)]{unbound} \label{12c2k}
Let $K$ be a nontrivial knot that satisfies the equality $u(K)=\dfrac{c(K)-1}{2}$.   Then $K$ is a $(2,\ p)$-torus knot for some odd number $p \neq \pm 1$.
\end{thm}

In this paper, as an extension of Proposition \ref{12c1k}, we show the following theorem.

\begin{thm}\label{12c1}
Let $H$ be a non-trivial handlebody-knot.\ Then $u(H) \leq \dfrac{c(H)-1}{2}$.
\end{thm}
The spine of genus 1 handlebody-knot is a classical knot. Therefore Theorem \ref{12c1} is an extention of Proposition \ref{12c1k}. It follows from Theorem \ref{handcuff} and Theorem \ref{theta} that for any non-trivial handlebody-knot $H$ with genus 2 twice the unknotting number of $H$ is less than or equal to the crossing number of $H$ minus one \ (see section 4). 

It follows from Theorem \ref{12c2k} that genus 1 handlebody-knot $H$ with $u(H) = \dfrac{c(H)-1}{2}$ is a regular neighborhood of a $(2,\ p)-$torus knot. We also characterize genus $n \geq 2$ handlebody-knots which satisfy the equality as follows.
\begin{thm}\label{12c2}
Let $n\geq 2$ and let $H$ be a nontrivial genus $n$ handlebody-knot that satisfies the equality $u(H)=\dfrac{c(H)-1}{2}$.\,Then $H$ is a handlebody-knot represented by $D_3$ or $D_{-3}$ illustrated in Fig.\ \ref{d3}.
\end{thm}
\begin{figure}[H]
      \begin{tabular}{c}
      \begin{minipage}{0.5\hsize}
        \begin{center}
          \includegraphics[width=0.6\linewidth]{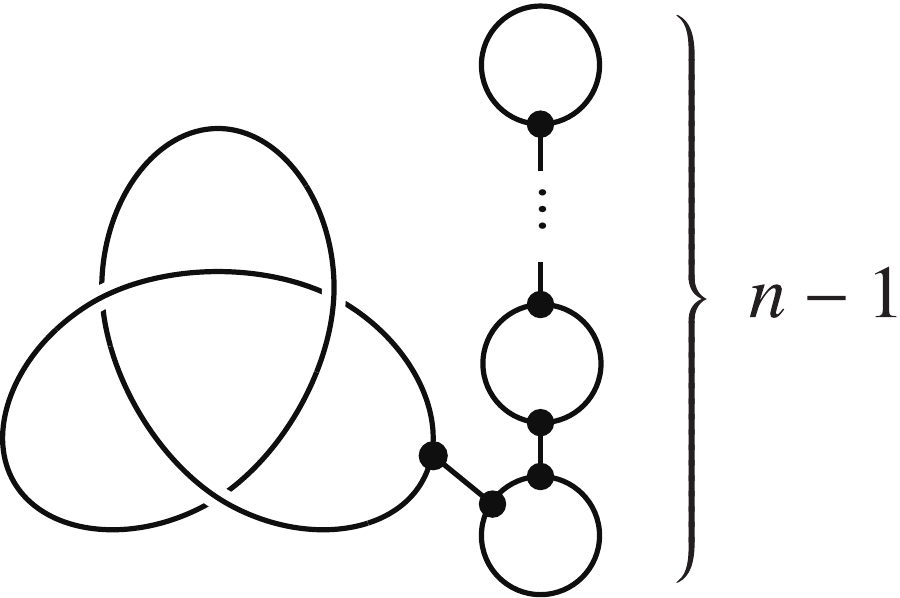}
        \end{center}
      \end{minipage}
      \begin{minipage}{0.5\hsize}
        \begin{center}
          \includegraphics[width=0.6\linewidth]{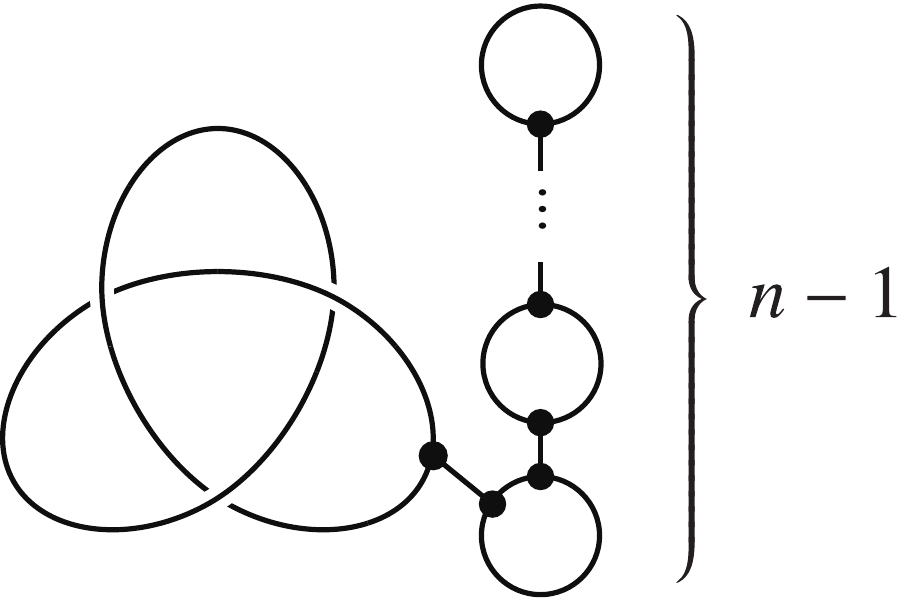}
        \end{center}
      \end{minipage}
      \end{tabular}
      \caption{}
      \label{d3}
\end{figure}

This paper consists of five sections. In section 2 we review trivializability of planar graphs and  inequalities between unknotting numbers and crossing numbers of spatial embeddings of planar graphs. In section 3 we introduce unknotting number of handlebody-knots. In section 4 we give proofs of Theorem \ref{handcuff} and Theorem \ref{theta}. In section 5 we give proofs of Theorem \ref{12c1} and Theorem \ref{12c2}.

\section{UNKNOTTING NUMBERS AND CROSSING NUMBERS OF SPATIAL EMBEDDINGS OF PLANAR GRAPHS}

Let $G$ be a planar graph. A \textit{spatial embedding} of $G$ is an embedding $f : G\rightarrow \mathbb{R}^3$. Its image $f(G)$ is said to be a \textit{spatial graph}. Let $\pi: \mathbb{R}^3 \rightarrow \mathbb{R}^2$ be a natural projection defined by $\pi(x,y, z) = (x, y)$. Let $SE(G)$ be the set of all spatial embeddings of $G$. A \textit{regular projection} of $G$ is a continuous map $\tilde{f} : G \rightarrow \mathbb{R}^2$ whose double points are only finitely many transversal double points. Such a double point is said to be a \textit{crossing point} or simply a \textit{crossing}.  If we give over/under informations at each crossing points of a regular projection $\tilde{f}$ of $G$, then $\tilde{f}$ together with the over/under informations represents a spatial embedding $f:G \rightarrow \mathbb{R}^3$ such that $\tilde{f} = \pi \circ f$. Such a regular projection together with the over/under informations is said to be a \textit{diagram} of $f(G)$. Then we say that $f$ is \textit{obtained from} $\tilde{f}$. We also call $\tilde{f}$ a \textit{regular projection} of $f(G)$. For a diagram $D$ of a spatial embedding, the set of all crossings of $D$ is denoted by $\textit{C}(D)$. The number of crossings of $D$ is denoted by $c(D) = |\textit{C}(D)|$.

An element $f \in SE(G)$ is said to be \textit{trivial}, if it is ambient isotopic to  $t \in SE(G)$ such that $t(G) \subset \mathbb{R}^2$. 
Any spatial embedding of a planar graph can be transformed into trivial one by crossing changes. Therefore unknotting number is naturally extended to spatial embeddings of planar graphs as follows. For $f \in SE(G)$, the unknotting number $u(f)$ is defined to be the minimal number of crossing changes from $f$ to a trivial embedding of $G$. The \textit{crossing number} $c(f)$ is defined to be the minimal number of crossing points among all diagrams of spatial embeddings that are ambient isotopic to $f$. 

For any link $L$, $L$ satisfies the inequality $u(L) \leq \dfrac{c(L)}{2}$. But this is not extended for spatial embeddings of planar graph, namely there are a planar graph $G$ and a spatial embedding $f$ of $G$ such that $u(f) > \dfrac{c(f)}{2}$. Let $P_3$ the cube graph and $f_3 \in SE(P_3)$ a spatial embedding of $P_3$ as illustrated in Fig.\ \ref{f3p3}. The spatial graph $f_3(P_3)$ contains three Hopf-links and one crossing change of edges of $f_3(P_3)$ unknot at most two of them\ (See Fig.\ \ref{hopf3}). Then $u(f_3) \geq 2$. Since $f_3(P_3)$ contains a trefoil whose crossing number is 3, $c(f_3)=3$ and $u(f_3) > \dfrac{c(f_3)}{2}$ \cite{uncr}. 


\begin{figure}[H]
\centering
\includegraphics[width=0.8\linewidth]{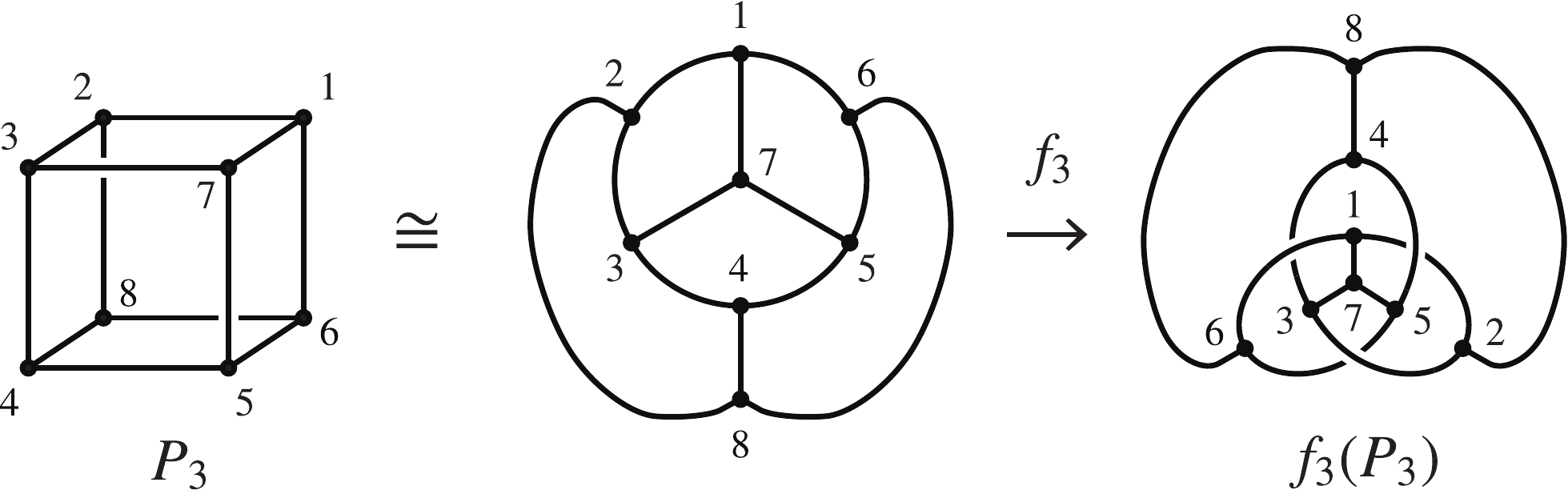}\\
\caption{}
\label{f3p3}
\end{figure}

\begin{figure}[H]
\centering
\includegraphics[width=0.55\linewidth]{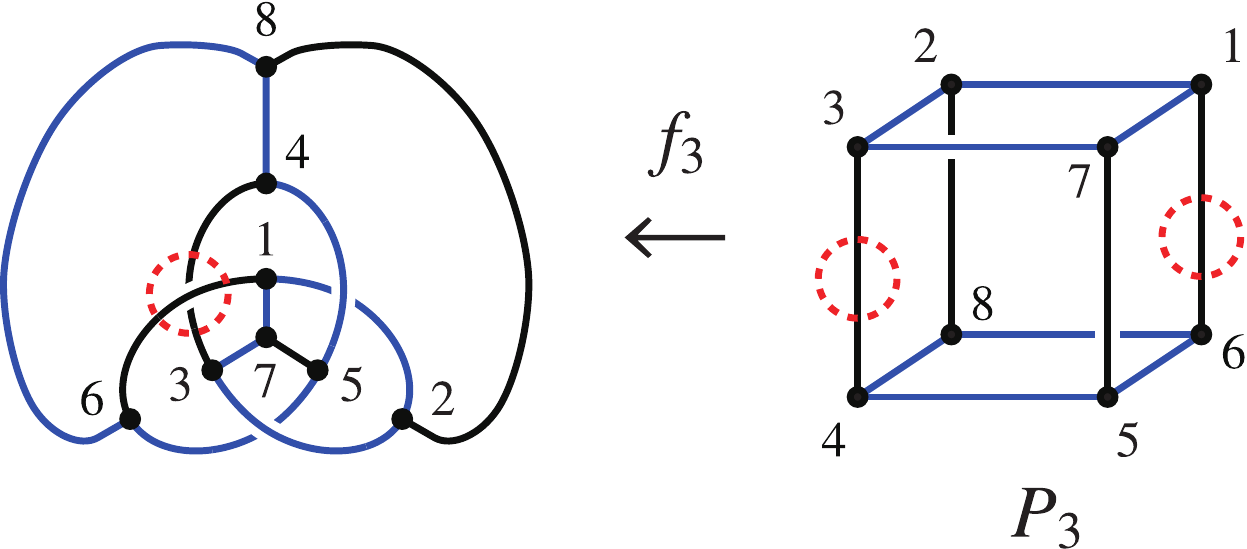}
\caption{}
\label{hopf3}
\end{figure}

Now we review the reason why it happens for some planar graphs. The key point of the proof of $u(L) \leq \dfrac{c(L)}{2}$ for a link $L$ is that any link diagram can be transformed into a trivial link diagram by changing over/under informations at some crossings of the diagram. Let $D$ be a minimal crossing diagram of $L$. Let $A$ be a subset of $C(D)$ such that changing over/under informations at all crossings in $A$ turns $D$ to a diagram $T_1$ of a trivial link. Let $T_2$ be a diagram that is obtained from $T_1$ by changing over/under informations at all crossings. A mirror image of a trivial link is also  trivial.  Thus $T_2$ is a diagram of a trivial link. Note that $T_2$ is obtained from $D$ by changing over/under informations at all crossings in $C(D)-A$. Therefore we have
$$u(L) \leq u(D) \leq \min  \{|A|,|C(D)-A|\} \leq \dfrac{c(D)}{2}=\dfrac{c(L)}{2}$$

On the other hand, all diagrams obtained from $\pi \circ f_3(P_3)$ (Fig.\ \ref{knotted}) represent non-trivial spatial graphs since each of the spatial graphs obtained from these diagrams contains at least one Hopf-link. A regular projection $\tilde{f}$ of a planar graph $G$ is said to be a \textit{knotted projection} \cite{knotted}, if all spatial embeddings of $G$ which can be obtained from $\tilde{f}$ are non-trivial. 



\begin{figure}[H]
        \begin{center}
          \includegraphics[height=0.18\linewidth]{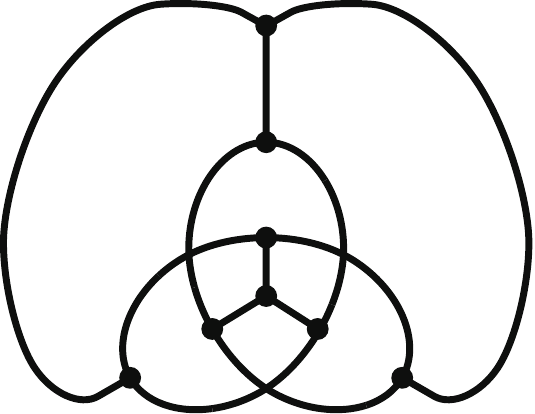}\\
          \ \\
          $\pi \circ f_3(P_3)$
        \end{center}
\caption{}
\label{knotted}
\end{figure}

A planar graph is said to be \textit{trivializable} if it has no knotted projections.  In \cite{knotted} Taniyama gave a class of trivializable graphs. In \cite{class} Sugiura and Suzuki extended the class. In \cite{tamura} Tamura gave another class of trivializable graphs.

For a spatial embedding of a trivializable planar graph, the same argument as
for a link works, and we have the following proposition.

\begin{prop}\label{tri12} \cite{uncr}
Let $G$ be a trivializable planar graph and $f: G \rightarrow \mathbb{R}^3$ a spatial embedding of $G$. Then $u(f) \leq \dfrac{c(f)}{2}$.
\end{prop}

\newpage

\section{UNKNOTTING NUMBERS AND CROSSING NUMBERS OF HANDLEBODY-KNOTS}

We review that crossing change of a handlebody-knot is an unknotting operation \cite{unqc}.

A \textit{diagram} of a handlebody-knot $H$ is that of a spatial trivalent graph representing $H$. In \cite{handle}, Ishii gave a list of fundamental moves among diagrams of handlebody-knots, which is called R1-6 moves illustrated in Fig.\ \ref{reide}. Ishii showed that two handlebody-knots are equivalent if and only if their representing diagrams are related by a finite sequence of R1-6 moves. Note that R6-move is also called \textit{IH-move}. 

\begin{figure}[H]
      \begin{tabular}{c}
      \begin{minipage}{0.3\hsize}
        \begin{center}
          \includegraphics[height=0.4\linewidth]{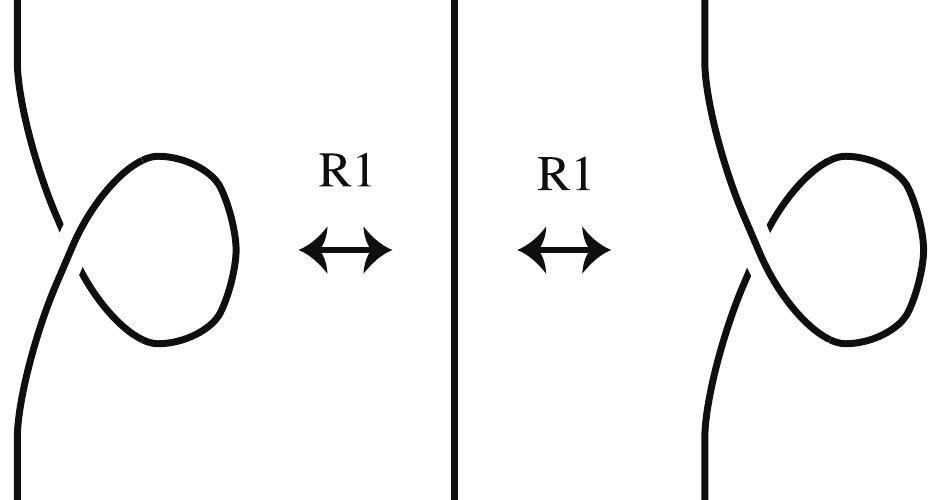}
        \end{center}
      \end{minipage}
      \begin{minipage}{0.25\hsize}
        \begin{center}
          \includegraphics[height=0.48\linewidth]{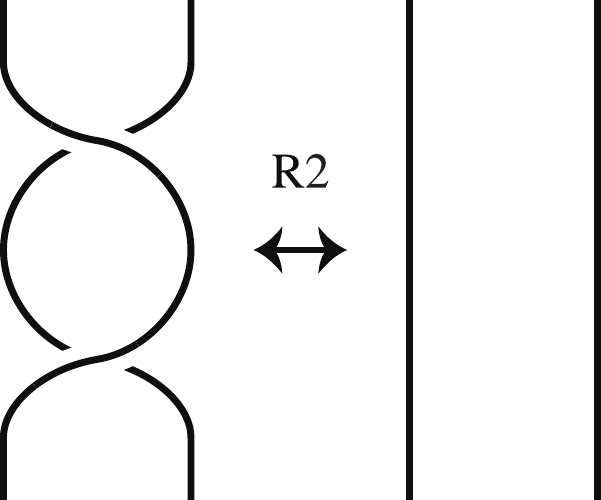}
        \end{center}
      \end{minipage}
      \begin{minipage}{0.25\hsize}
        \begin{center}
          \includegraphics[height=0.48\linewidth]{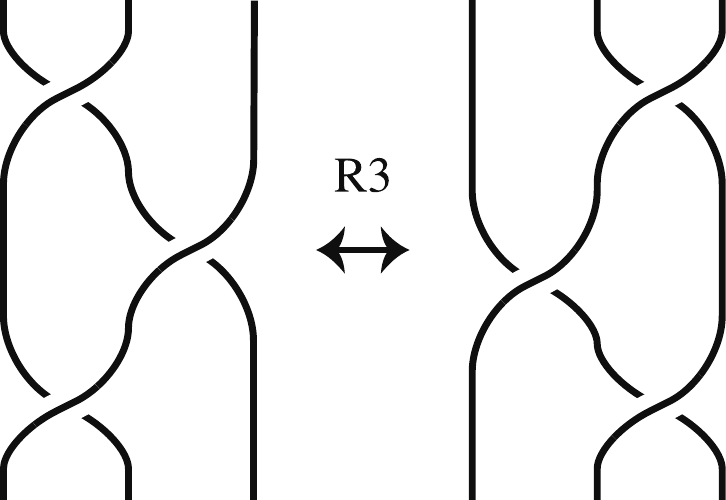}
        \end{center}
      \end{minipage}
         \end{tabular}\\
\ \vspace{0.5cm} \\
       \begin{tabular}{c}
      \begin{minipage}{0.36\hsize}
        \begin{center}
          \includegraphics[height=0.333333\linewidth]{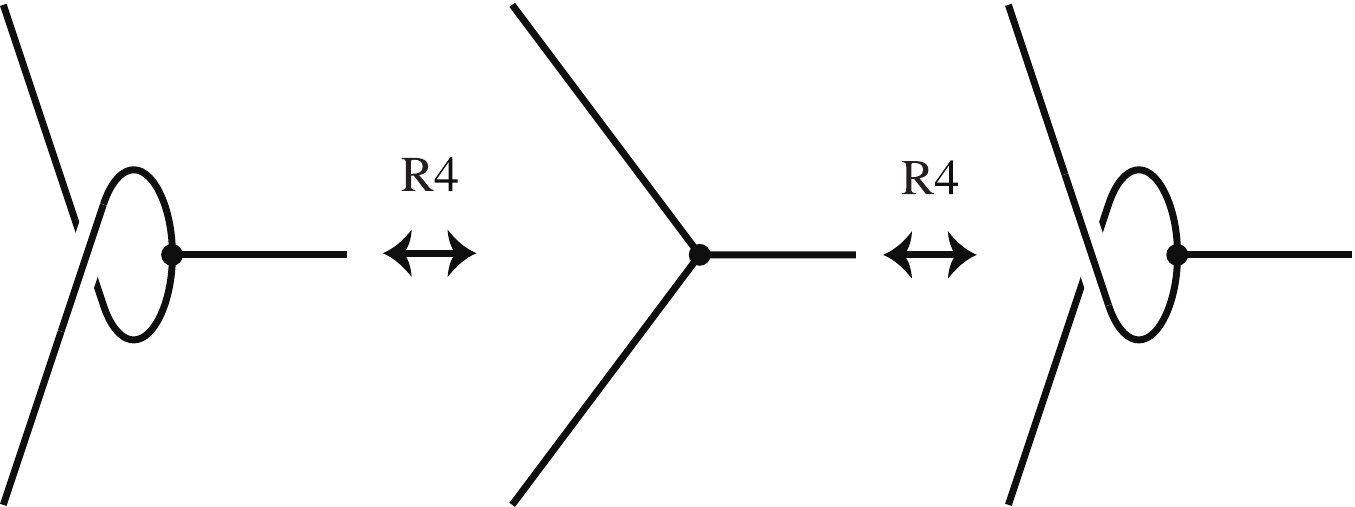}
        \end{center}
      \end{minipage}
      \begin{minipage}{0.36\hsize}
        \begin{center}
          \includegraphics[height=0.333333\linewidth]{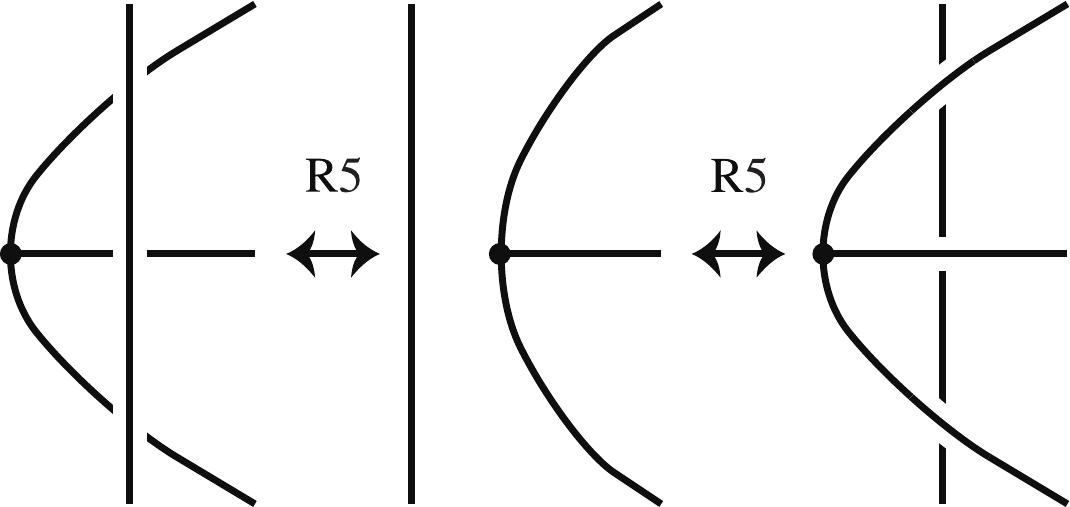}
        \end{center}
      \end{minipage}
      \begin{minipage}{0.2\hsize}
        \begin{center}
          \includegraphics[height=0.6\linewidth]{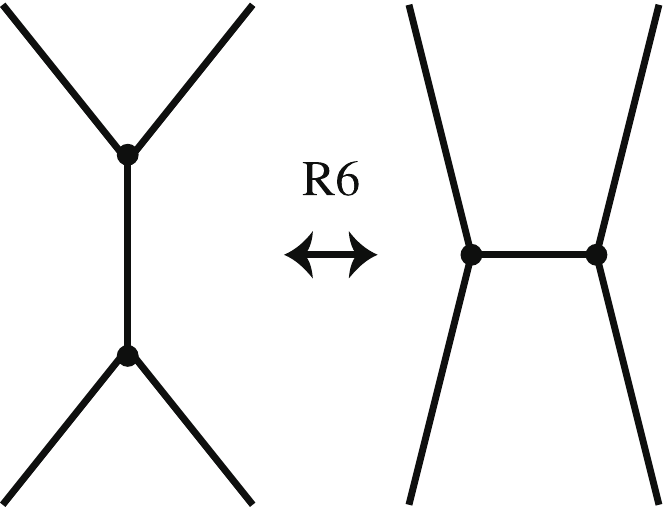}
        \end{center}
      \end{minipage}
      \end{tabular}
      \caption{}
      \label{reide}
\end{figure}

A crossing change of a handlebody-knot $H$ is that of a spatial trivalent graph representing $H$. This move can be realized by switching two tubes illustrated in Fig.\ \ref{cch}. A genus $n$ handlebody-knot is \textit{trivial} if it is equivalent to a handlebody-knot represented by a diagram  illustrated in Fig.\ \ref{tg}.


\begin{figure}[H]
\centering
          \includegraphics[height=0.12\linewidth]{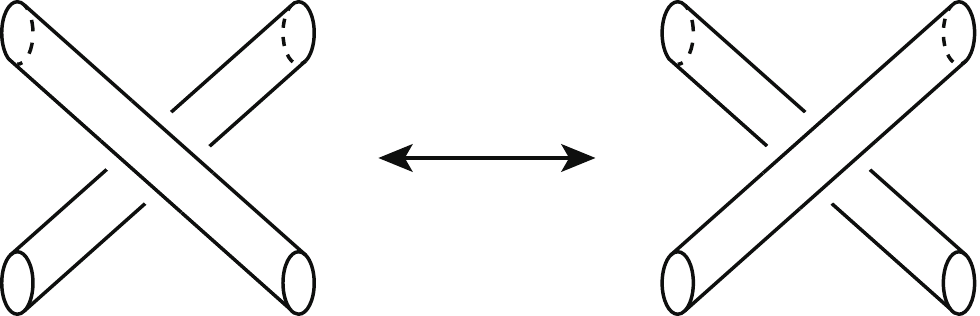}
         \caption{}
         \label{cch}
\end{figure}


\begin{figure}[H]
      \centering
      \includegraphics[width=0.4\linewidth]{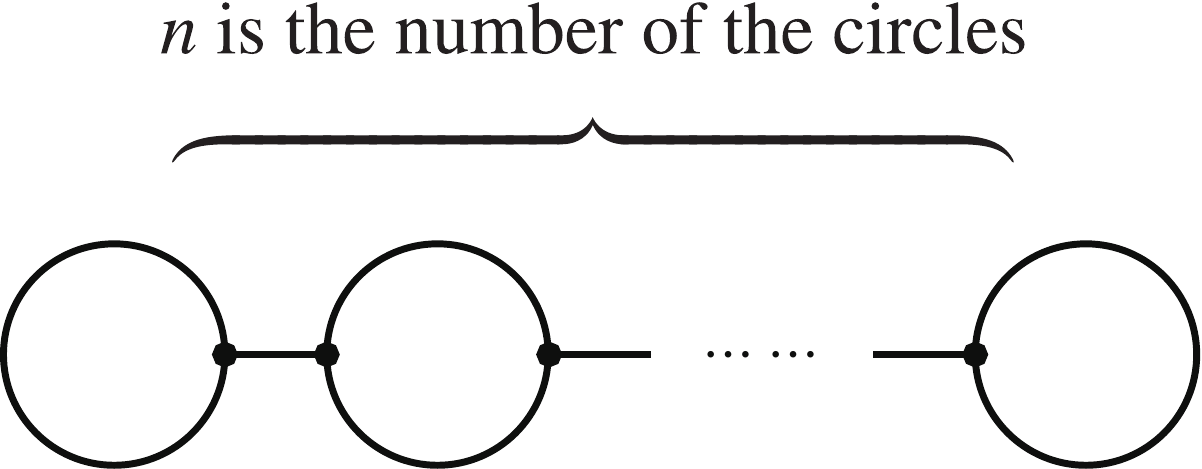}
         \caption{}
         \label{tg}
\end{figure}

Let $T_n$ be the trivalent graph whose image is illustrated in Fig.\ \ref{tg}. Any handlebody-knot is represented by a diagram of a spatial embedding of $T_n$ since a genus $n$ handlebody has $T_n$ as a spine. Note that $T_n$ is a trivializable graph \cite{class}. Namely, any diagram $D$ of a spatial embedding of $T_n$ can be changed to a trivial spatial graph diagram by changing over/under informations at some crossings of $D$. Then we have the following proposition.

\begin{prop}\label{unop} \cite[Proposition\ 2.1]{unqc} Any handlebody-knot can be transformed into trivial one by crossing changes.
\end{prop}

Therefore unknotting number is naturally extended to handlebody-knots as follows. For a handlebody-knot $H$, the unknotting number $u(H)$ is the minimal number of crossing changes needed to obtain a trivial handlebody-knot from $H$. The crossing number $c(H)$ is the minimal number of crossing points among all diagrams of handlebody-knots that are equivalent to $H$. 

By the proof of \cite[Proposition\ 3.1]{forbidden} we see that any diagram $D$ of a spatial graph can be transformed into a diagram of a spatial graph whose neighborhood are ambient isotopic to a neighborhood of a trivial bouquet by changing over/under informations at some crossings of $D$. Therefore, in \cite{unqc}, Iwakiri also showed that Proposition\ \ref{unop} can be refined to the strong statement as follows.

\begin{prop}\label{dunop} \cite{unqc} Any handlebody-knot diagram can be transformed into a trivial handlebody-knot diagram by changing over/under informations at some crossings of the diagram.
\end{prop}



For a handlebody-knot diagram $D$, the unknotting number $u(D)$ is the minimal number of changing over/under informations at crossings of $D$ needed to obtain a trivial handlebody-knot diagram. Same as Proposition\ \ref{tri12} we have $u(D) \leq \dfrac{c(D)}{2}$ and $u(H) \leq \dfrac{c(H)}{2}$.

In section 5, we show that a handlebody-knot $H$ satisfies $u(H)=\dfrac{c(H)}{2}$ if and only if $H$ is trivial\ (Theorem\ \ref{12c1}). Then it is natural to ask when handlebody-knots satisfy the equality $u(H) =\dfrac{c(H)-1}{2}$. Let $H$,\ $H_1$ and $H_2$ be handlebody-knots in $\mathbb{R}^3$ and let $S$ be a $2-$sphere in $\mathbb{R}^3$. Suppose that $H \cap S=H_1 \cap H_2$ is a $2-$disk and $H=H_1 \cup H_2$. Then $H$ is said to be a \textit{disk sum} of $H_1$ and $H_2$ and denoted by $H=H_1 \# H_2$. In \cite{unbound} Taniyama showed that if a classical knot $K$ satisfies $u(K) = \dfrac{c(K)-1}{2}$ then $K$ is a $(2,\ p)$-torus knot for some odd number $p \neq \pm 1$\ (Theorem\ \ref{12c2k}). Therefore the handlebody-knots illustrated in Fig.\ \ref{d2n1} may satisfy the equality. But by the following proposition only two of these handlebody-knots satisfy the equality.

\begin{prop}\label{2br}
Let $n\geq 2$ and let $H$ be a genus $n$ handlebody-knot such that $H=K\ \#\ O_{n-1}$, where $K$ is a genus $1$ handlebody-knot whose spine is a $2-$bridge knot and $O_{n-1}$ is a genus $n-1$ trivial handlebody-knot. Then $u(H)=1$.
\end{prop}

\begin{figure}[H]
\begin{tabular}{c}

      \begin{minipage}[t]{0.24\hsize}
        \begin{center}
          \includegraphics[height=0.6\linewidth]{d311.eps}\\
          $D_3$
        \end{center}
      \end{minipage}
      
      \begin{minipage}[t]{0.24\hsize}
        \begin{center}
          \includegraphics[height=0.6\linewidth]{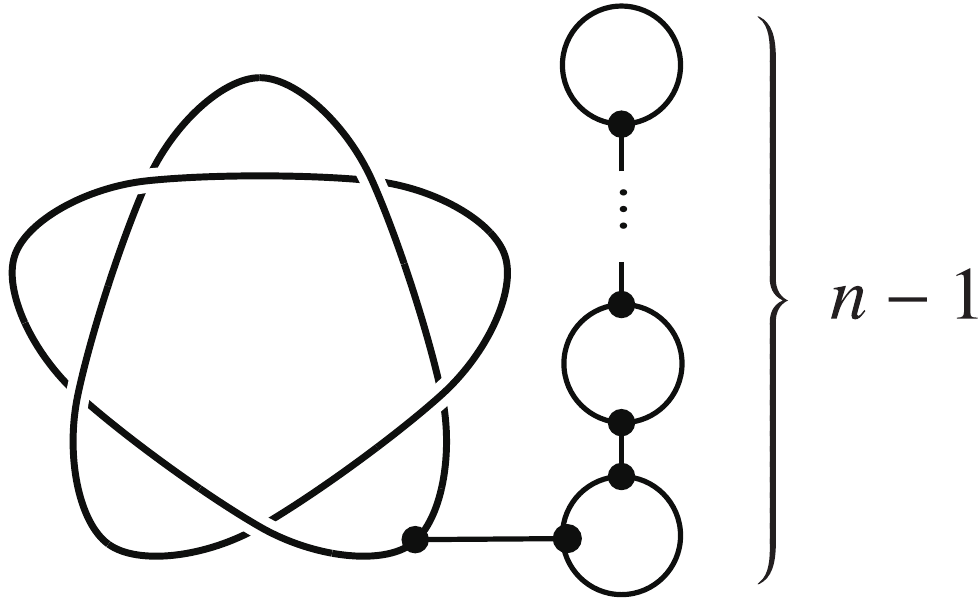}\\
          $D_5$
        \end{center}
      \end{minipage}
      
        \begin{minipage}[t]{0.32\hsize}
         \begin{center}
          \includegraphics[height=0.45\linewidth]{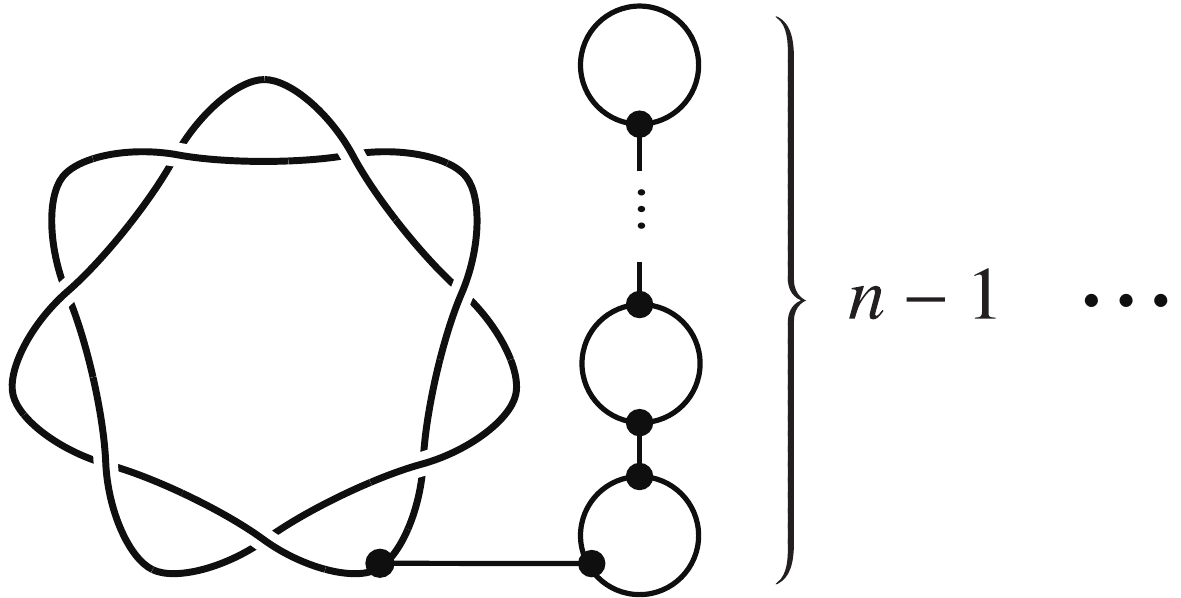}\\
          $D_7$
        \end{center}
      \end{minipage}
      \end{tabular}\\
      \ \vspace{0.5cm} \\
\begin{tabular}{c}

      \begin{minipage}[t]{0.24\hsize}
        \begin{center}
          \includegraphics[height=0.6\linewidth]{d312.eps}\\
          $D_{-3}$
        \end{center}
      \end{minipage}
      
      \begin{minipage}[t]{0.24\hsize}
        \begin{center}
          \includegraphics[height=0.6\linewidth]{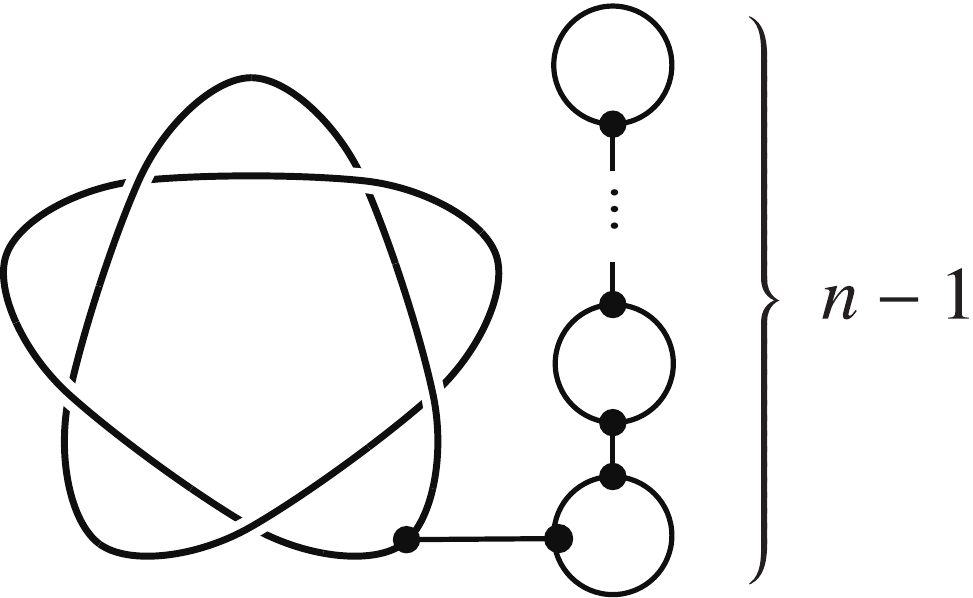}\\
          $D_{-5}$
        \end{center}
      \end{minipage}
      
        \begin{minipage}[t]{0.32\hsize}
         \begin{center}
          \includegraphics[height=0.45\linewidth]{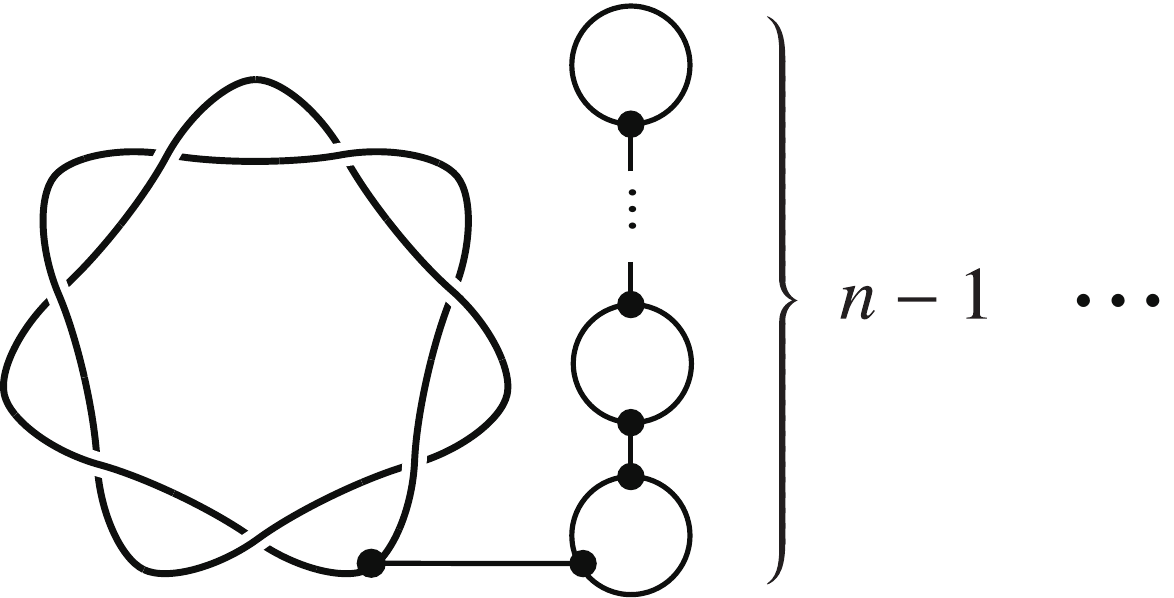}\\
          $D_{-7}$
        \end{center}
      \end{minipage}
      \end{tabular}
      \caption{}
      \label{d2n1}
\end{figure}

\begin{proof}
Let $K'$ be the spine of $K$. Let $H'$ be the handlebody-knot obtained from $K\# O_{n-1}$ by one crossing change as illustrated in the left of Fig.\ \ref{dut}. By \cite[Proposition\ 3.1]{tunnel} we see that the tunnel $\tau$ for $K'$ as illustrated in the right of Fig.\ \ref{dut} is an unknotting tunnel. Therefore the genus 2 handlebody-knot represented by the right of Fig.\ \ref{dut} is trivial. Since a disk sum of two trivial handlebody-knots is trivial,  $H'$ is also trivial.

\begin{figure}[H]
      \begin{tabular}{c}
      \begin{minipage}[t]{0.5\hsize}
        \begin{center}
          \includegraphics[height=0.95\linewidth]{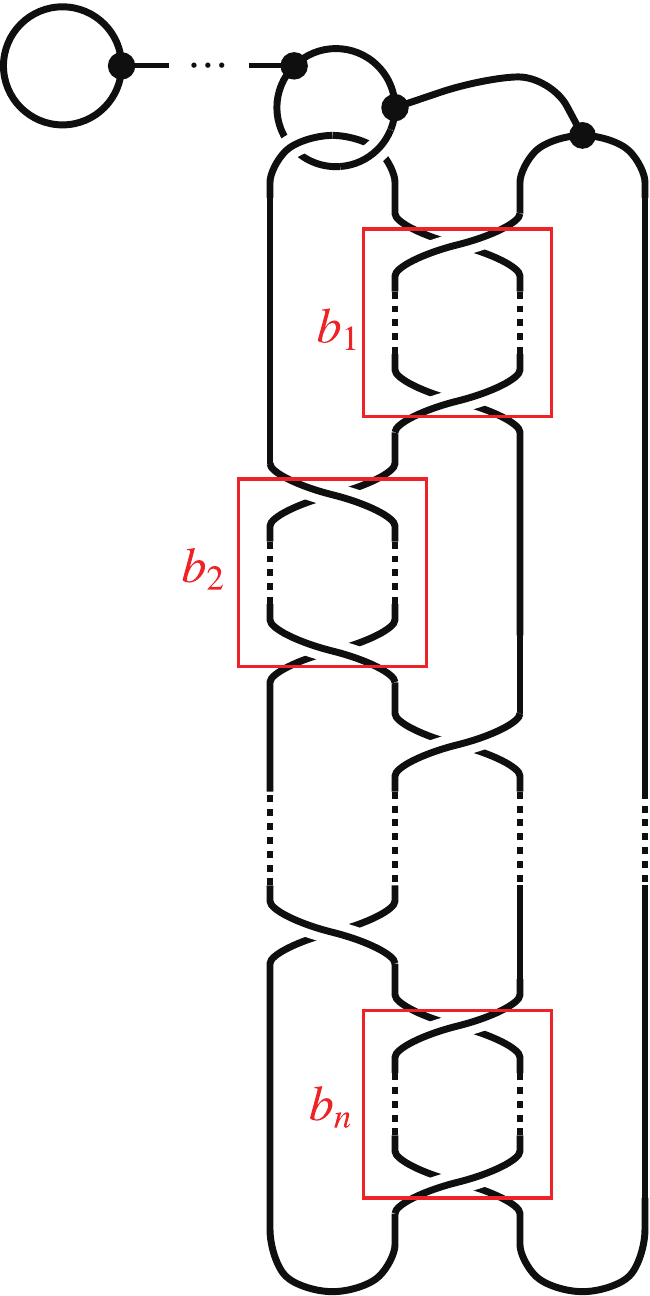} \vspace{0.5cm}\\
          \hspace{0.2\linewidth}$b_1,\ b_2,\ \cdots b_n:\ 2$-braids
        \end{center}
        
      \end{minipage}
      \begin{minipage}[t]{0.5\hsize}
        \begin{center}
          \includegraphics[height=0.9\linewidth]{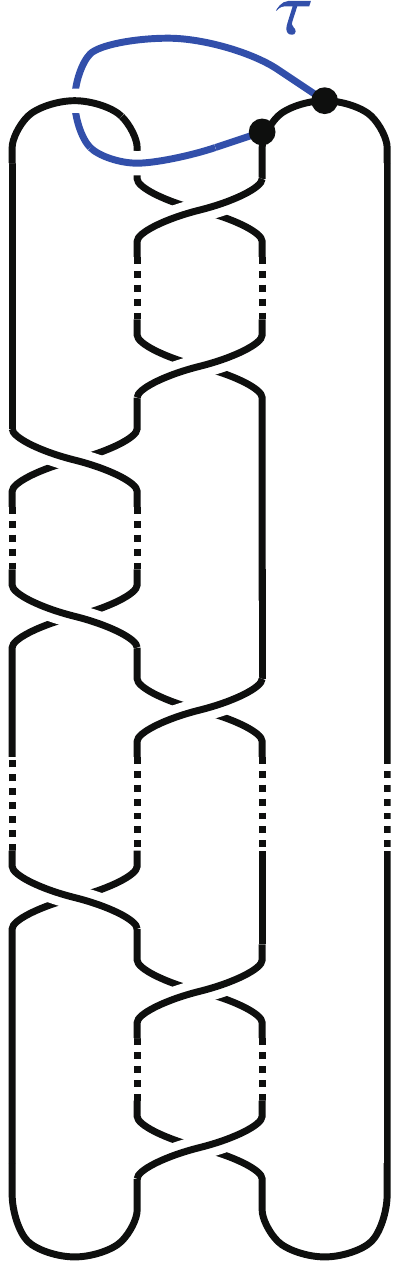}
        \end{center}
      \end{minipage}
      \end{tabular}
      \caption{}
      \label{dut}
\end{figure}

\end{proof}


\section{PROOFS OF THEOREM\ \ref{handcuff} AND THEOREM\ \ref{theta}}
Let $D$ be a diagram of a spatial graph $f(G)$ and let $H$ be a subgraph of $G$. Then the diagram of $f(H)$ that is contained in $D$ is said to be a \textit{subdiagram} of $D$. For subdiagrams $A,\ B$ of a diagram $D$, let $c(A)$ be the number of all crossings on $A$ among the crossings of $D$ and let $c(A,\,B)$ be the number of all crossings between  $A$ and $B$. 

\begin{lemma}\label{self}
Let $G$ be a trivializable graph and let $f$ be a spatial embedding of $G$. Let $D$ be a diagram of $f(G)$. If $D$ has a self-crossing, then $u(D) \leq \dfrac{c(D)-1}{2}$.
\end{lemma}
\noindent
\begin{proof}
Let $P$ be a self-crossing of $D$. By smoothing $D$ at $P$, we have a diagram $D'$ such that one of the components of $D'$ represents a knot\ (see Fig.\ \ref{selfc}). Let $\gamma_1$ be a component of $D'$ that represents a knot and let $\gamma_2$ be the other component of $D'$. 

\begin{figure}[H]
      \begin{tabular}{c}
      \begin{minipage}[b]{0.5\hsize}
        \begin{center}
          \includegraphics[width=0.5\linewidth]{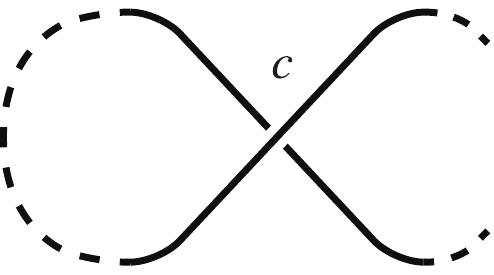} \vspace{0.35cm} \\
          \hspace{0.08\linewidth}$D$
        \end{center}
        
      \end{minipage}
      \begin{minipage}[b]{0.5\hsize}
        \begin{center}
          \includegraphics[width=0.5\linewidth]{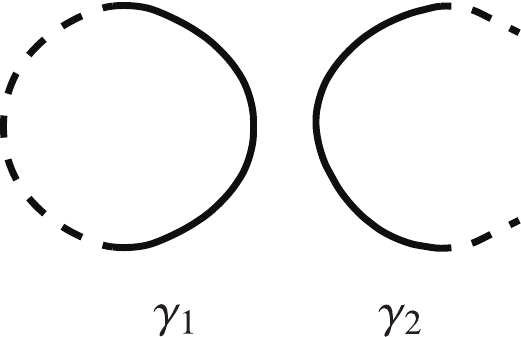}\\
          \hspace{0.08\linewidth}$D'$
        \end{center}
      \end{minipage}
      \end{tabular}
      \caption{}
      \label{selfc}
\end{figure}

If we change some crossings on $\gamma_1$ so that the part $\gamma_1$ is over other component of $D$ and itself unknotted then we have a spatial embedding that has a diagram $\gamma_2$. Also we may change some crossings on $\gamma_1$ so that the part $\gamma_1$ is under the other component of $D$ and itself unknotted. Note that we can choose these two crossing changes complementary on the crossings on $\gamma_1$. We choose one of them that have no more crossing changes than the other. Thus by changing no more than $\dfrac{c(D)-c(\gamma_2)-1}{2}$ crossings of $D$ we have a spatial embedding that has a diagram $\gamma_2$. Note that the key point here is that we do not need to change the crossing $c$. Since $\gamma_2$ is also a diagram of a spatial embedding of a trivializable graph, we have $u(\gamma_2) \leq \dfrac{c(\gamma_2)}{2}$. Therefore we have 
$$u(D) \leq u(\gamma_2) + \dfrac{c(D)-c(\gamma_2)-1}{2} \leq \dfrac{c(\gamma_2) +(\ c(D)-c(\gamma_2)-1\ )}{2}=\dfrac{c(D)-1}{2}$$
\end{proof}

\begin{lemma}\label{mini}
Let $G$ be a trivializable planar graph and let $f$ be a spatial embedding of $G$ such that $u(f)=\dfrac{c(f)}{2}$. Let $D$ be a minimal crossing diagram of $f(G)$. Then $u(D)=\dfrac{c(D)}{2}$.
\end{lemma}

\begin{proof} It is sufficient to show that $u(D) \geq \dfrac{c(D)}{2}$. Since $u(f) \leq u(D)$ and $c(f)=c(D)$ we have
$$u(D) \geq u(f)=\dfrac{c(f)}{2}=\dfrac{c(D)}{2}$$
\end{proof}

\begin{lemma}\label{dhandcuff}
Let $D$ be a diagram of a spatial embedding of a handcuff-graph such that $u(D)=\dfrac{c(D)}{2}$. Then $D$ satisfies the following conditions {\rm:} \\
$(1)\,$Each edge of $D$ has no self-crossings. \\
$(2)\,$All crossings of $D$ are crossings between two loops.\\
$(3)\,$Two loops of $D$ form an alternating diagram or a diagram without crossings.
\end{lemma}

\begin{proof}
By Lemma \ref{self} $D$ satisfies $(1)$. In the following we show that $D$ satisfies $(2)$ and $(3)$. 

Let $\gamma_1$ and $\gamma_2$ be two loops of $D$ and let $e$ be the edge of $D$ that is not $\gamma_i$\ ($i=1,\ 2$).  If we change some crossings on $\gamma_2$ so that the part $\gamma_2$ is over $D-\gamma_2$ of $D$ then we have a diagram of a trivial spatial embedding of $G$ since $\gamma_i$ is a simple closed curve on $\mathbb{R}^2$\ ($i=1,\ 2$). See for example Fig.\ \ref{alc}.  Also we may change some crossings on $\gamma_2$ that the part $\gamma_2$ is under $D-\gamma_2$ of D and itself unknotted. Note that these two crossing changes are complementary on the crossings on $\gamma_2$. We choose one of them that have no more  crossing changes than the other. Thus by changing no more than $\dfrac{c(D)-c(\gamma_1,\ e)}{2}$ crossings of $D$ we have a trivial diagram and $u(D) \leq \dfrac{c(D)-c(\gamma_1,\ e)}{2}$. The key point here is that we do not need to change crossings between $\gamma_1$ and $e$.  Since $u(D) = \dfrac{c(D)}{2}$ we have $c(\gamma_1,\ e)=0$. Similarly we have $c(\gamma_2,\ e)=0$. Therefore $D$ satisfies $(2)$.

\begin{figure}[H]
      \begin{tabular}{c}
      \begin{minipage}[b]{0.3\hsize}
        \begin{center}
          \includegraphics[height=0.7\linewidth]{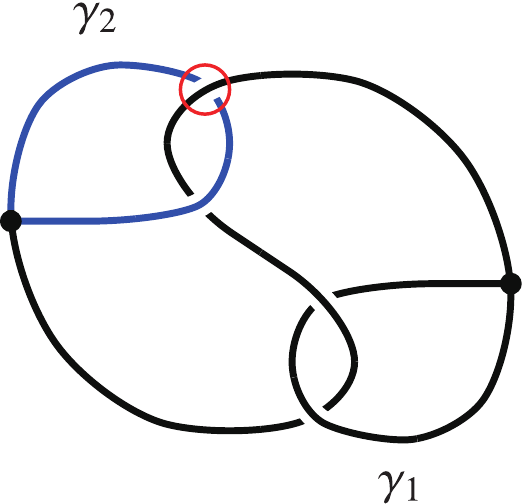}
        \end{center}
        
      \end{minipage}
      \begin{minipage}[b]{0.7\hsize}
        \begin{center}
          \includegraphics[height=0.3\linewidth]{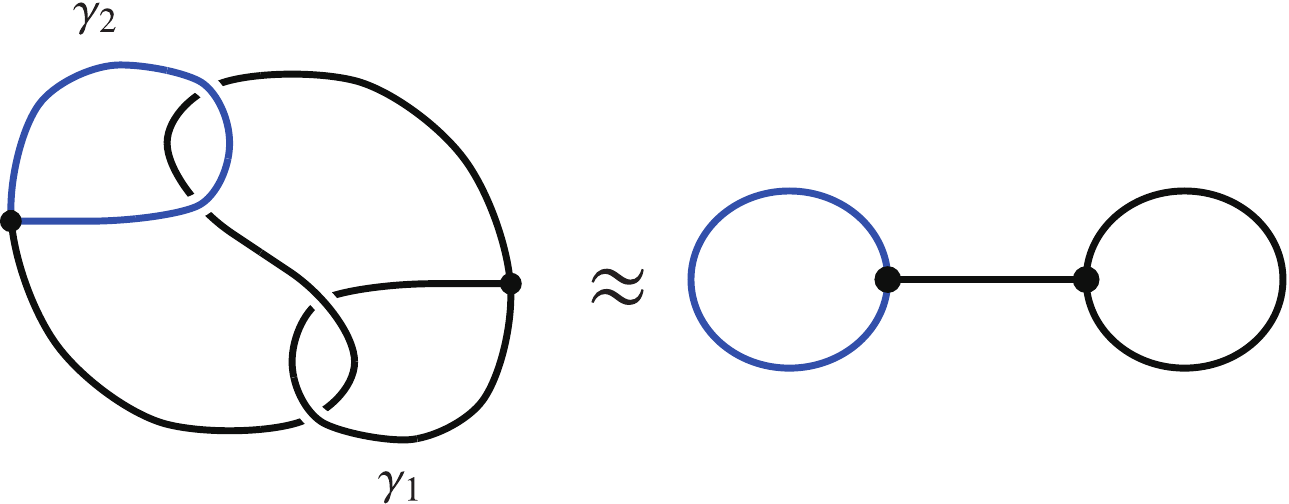}
        \end{center}
      \end{minipage}
      \end{tabular}
      \caption{}
      \label{alc}
\end{figure}

Suppose that $\gamma_1 \cup \gamma_2$ is not an alternating diagram. Then we may suppose without loss of generality that there is an arc $\alpha$ of $\gamma_1$ disjoint from $e$ such that $\alpha \cap \gamma_2=\partial \alpha=\{c_1,\ c_2\}$ and $\gamma_1$ is over $\gamma_2$ at $c_1$ and $c_2$. See Fig.\ \ref{disarc} .

\begin{figure}[H]
\centering
\includegraphics[width=0.25\linewidth]{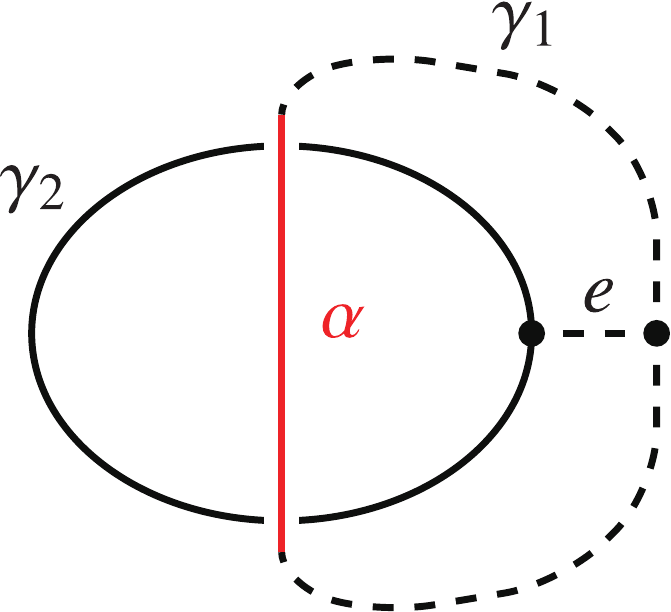}
\caption{}
\label{disarc}
\end{figure}

Let $A$ be the set of all crossings of $D$ at which $\gamma_1$ is under $\gamma_2$. Let $B=C(D)\backslash (A \cup \{c_1.\ c_2\})$. Then by the height function argument first used in \cite{knotted} we see that changing all crossings in $A$ (resp. $B$) produce a trivial spatial embedding. See Fig. \ref{althand}.

\begin{figure}[H]
\begin{tabular}{c}
      \begin{minipage}[b]{0.5\hsize}
        \begin{center}
          \includegraphics[width=0.8\linewidth]{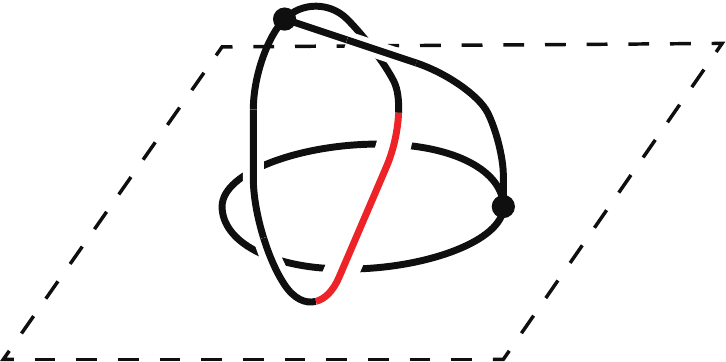} \vspace{0.55cm}

        \end{center}
        
      \end{minipage}
      \begin{minipage}[b]{0.5\hsize}
        \begin{center}
          \includegraphics[width=0.75\linewidth]{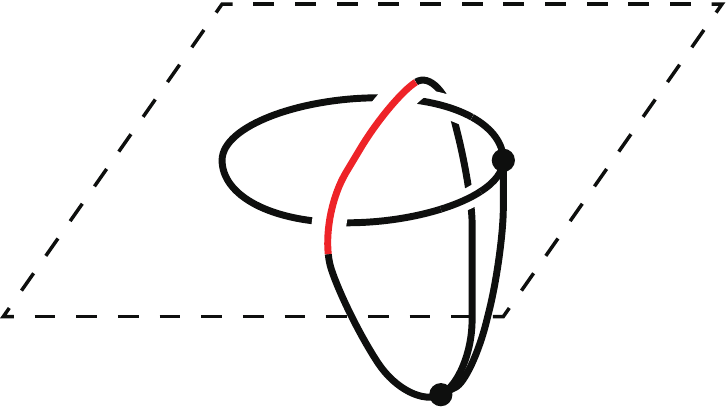} 

        \end{center}
      \end{minipage}
      \end{tabular}
      \caption{}
      \label{althand}
\end{figure}

Therefore we have
$$u(D) \leq \min \{|A|,\ |B|\} \leq \dfrac{c(D)-2}{2}$$
This is contradicts to the equation $u(D) = \dfrac{c(D)}{2}$. Thus $D$ satisfies $(3)$ as desired. 
\end{proof}

\noindent
\textit{Proof of Theorem\ref{handcuff}}\\
First, we show that if there exists a diagram $D$ of $f(G)$ satisfying $(1)$, $(2)$ and $(3)$, then $u(f) = \dfrac{c(f)}{2}$. We may suppose that $c(D)>0$. Let $L=l_1 \cup l_2$ be a $2-$component link represented by two loops of $D$. See for example Fig.\ \ref{altlink}. Since the diagram of $L$ consists of two simple closed curves and it is alternating, we see that twice the absolute value of the linking number $2|\,lk(l_1,l_2)\,|$ is equal to $C(D)$. 
 Therefore we have 
$$u(f) \geq u(L) \geq |\,lk(l_1,l_2)\,| = \dfrac{c(D)}{2}=\dfrac{c(f)}{2}$$
By Proposition \ref{tri12} we have $u(f)=\dfrac{c(f)}{2}$. 

\begin{figure}[H]
\centering
\includegraphics[width=0.25\linewidth]{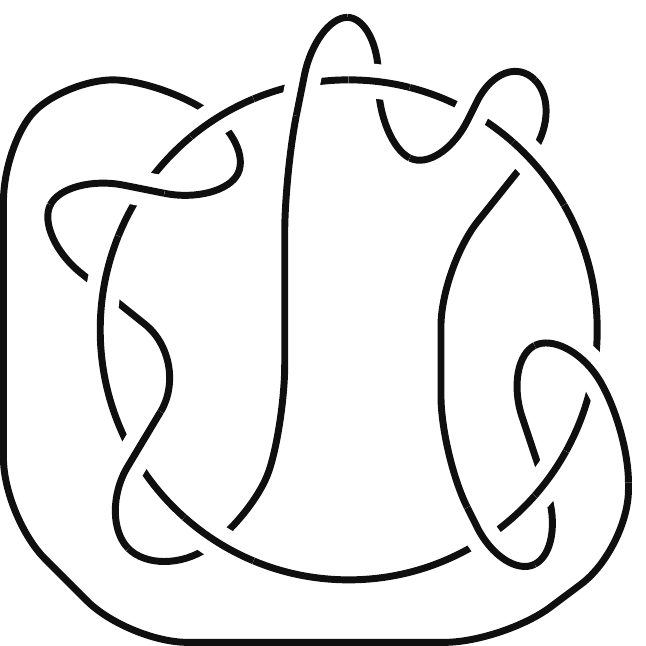}
\caption{}
\label{altlink}
\end{figure}

Let $f$ be a spatial embedding of $G$ such that $u(f)=\dfrac{c(f)}{2}$ and let $D$ be a minimal crossing diagram of $f(G)$. By Lemma\ \ref{mini} we have $u(D)=\dfrac{c(D)}{2}$. By Lemma \ref{dhandcuff} $D$ satisfies $(1)$, $(2)$ and $(3)$ as desired. \qed

\begin{lemma}\label{dtheta}
Let $G$ be a theta curve. Let $D$ be a diagram of a spatial embedding of $G$ such that $u(D)=\dfrac{c(D)}{2}$. Then $c(D)=0$.
\end{lemma}

\begin{proof}
By Lemma \ref{self} we may suppose that each edge of $D$ has no self-crossings. Suppose that $c(D)>0$. Then there exists a crossing $c$ on $D$ between two edges. Let $\tilde{f}:G \rightarrow \mathbb{R}^2$ be a regular projection of $G$ where $D$ is obtained from $\tilde{f}(G)$. Let $v$ and $u$ be two vertices of $G$. Let $G'$ be the graph obtained by adding $2$ verticies $v_1,\ v_1'$ to $G$ such that $\tilde{f}(v_1)=\tilde{f}(v_1')=c$ and $v_1$\ (resp.\ $v_1'$) is contained in the over-arc\ (resp.\ the under-arc) at $c$. Let $P$ be the path from $v$ to $u$ that contains $v_1$. We fix a spanning tree $T$ of $G'$ that contains $P$\ (see for example Fig.\ \ref{subdthe}). Let $h: G' \rightarrow \mathbb{R}$ be a continuous function with the following properties : \\
$(1)\,$For each vertex $t$ of $G'$, $h(t)=-d_{T}(t,\ v)$. Here $d_{T}(t,v)$ be the number of edges of the path in $T$ joining $t$ and $v$.\\
$(2)\,h|_e$ is injective for each edge $e$ of $G'$\\
\begin{figure}[H]
      \begin{tabular}{c}
      \begin{minipage}[b]{0.5\hsize}
        \begin{center}
          \includegraphics[width=0.45\linewidth]{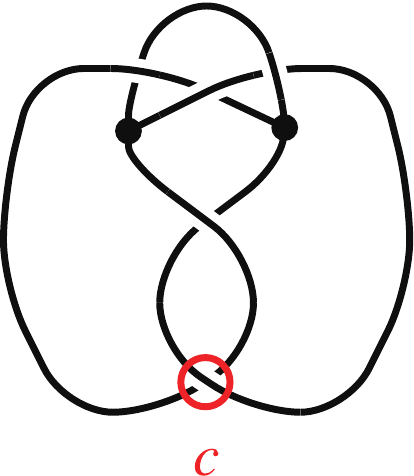} \vspace{0.3cm}\\
          $D$
        \end{center}
        
      \end{minipage}
      \begin{minipage}[b]{0.5\hsize}
        \begin{center}
          \includegraphics[width=0.7\linewidth]{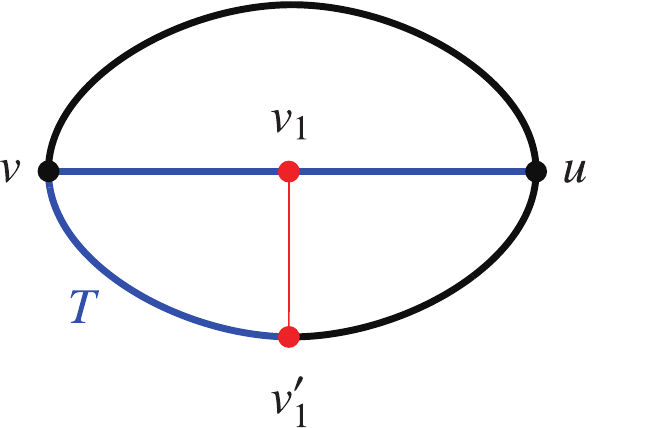} \vspace{0.3cm}\\
          $G'$
        \end{center}
      \end{minipage}
      \end{tabular}
      \caption{}
      \label{subdthe}
\end{figure}
We can deform $h$ slightly so that $h(v_1)>h(v_1')$ since $d_{T}(v,v_1)=d_{T}(v,v_1')=1$. Then we give over/under information to $\tilde{f}$ to produce a spatial embedding $f : G \rightarrow \mathbb{R}^3 = \mathbb{R}^2 \times \mathbb{R}$ such that $p_1 \circ f=\tilde{f}$ and $p_2 \circ f =h$, where $p_1$\ (resp. $p_2)$ denotes the projection of $\mathbb{R}^3$ to the first factor (respectively to the second factor) of $\mathbb{R}^2 \times \mathbb{R}$. Let $\Pi:\mathbb{R}^3 \rightarrow \mathbb{R}^2$ be a projection defined by $\Pi(x, y, z) = (x, z)$. We deform $f$ slightly by an ambient isotopy if necessary so that $\Pi \circ f$ is a regular projection. Then we can eliminate all crossings of $\Pi \circ f$ by  eliminating the crossing nearest to $v$ repeatedly (see Fig.\ \ref{heithe}). Therefore $f$ is trivial.
\begin{figure}[H]
\centering
\includegraphics[width=0.4\linewidth]{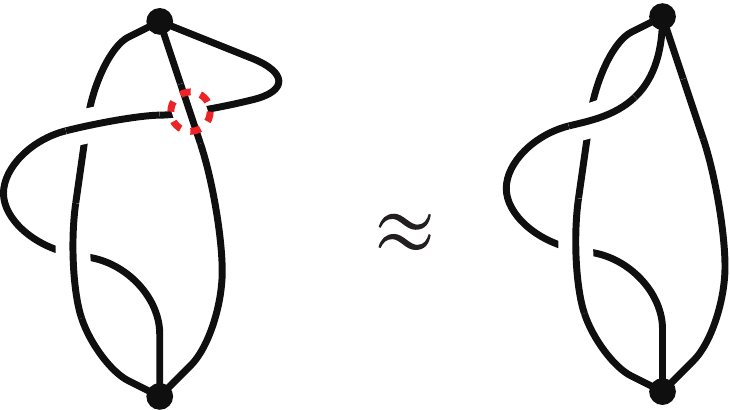}
      \caption{}
      \label{heithe}
\end{figure}
Let $D'$ be the diagram of $f(G)$ where $D'$ is obtained from $\tilde{f}(G)$. We note that $D$ and $D'$ are deformed  into each other by changing over/under informations of all crossing points without changing over/under informations of $c$. Let $D''$ be the diagram that is obtained from $D'$ by changing over/under informations of all crossing points with the exception of  $c$\ (see for example Fig.\ \ref{dandd}). Let $h':G' \rightarrow \mathbb{R}$ be a continuous function such that $h'=-h$. We can deform $h'$ slightly so that $h'(v_1)>h'(v_1')$. Then $D''$ is the diagram of a spatial embedding of $f' : G \rightarrow \mathbb{R}^3 = \mathbb{R}^2 \times \mathbb{R}$ such that $p_1 \circ f'=\tilde{f}$ and $p_2 \circ f'=h'$. Same as the case of $f(G)$, $f'$ is also trivial.
\begin{figure}[H]
      \begin{tabular}{c}
      \begin{minipage}[b]{0.5\hsize}
        \begin{center}
          \includegraphics[width=0.45\linewidth]{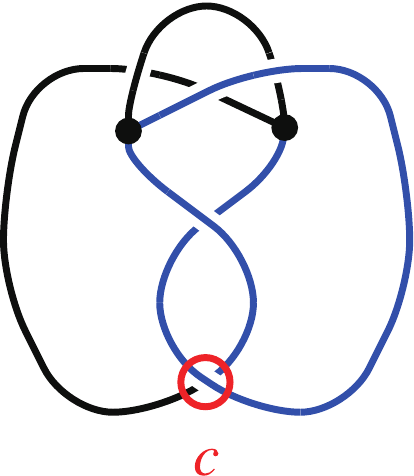} \vspace{0.3cm}\\
          $D'$
        \end{center}
        
      \end{minipage}
      \begin{minipage}[b]{0.5\hsize}
        \begin{center}
          \includegraphics[width=0.45\linewidth]{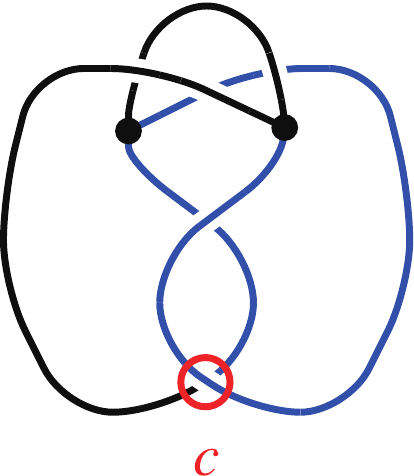} \vspace{0.3cm}\\
          $D''$
        \end{center}
      \end{minipage}
      \end{tabular}
      \caption{}
      \label{dandd}
\end{figure}

Let $A$ be a subset of $C(D)$ such that changing all crossings in $A$ turns $D$ to $D'$. We note that changing all crossings in $(\ C(D)-\{c\}\ )-A$ turns $D$ to $D''$. Therefore we have
$$u(D) \leq \min \{|A|,|(\ C(D)-\{c\}\ )-A|\} \leq \dfrac{c(D)-1}{2}$$ 

This is contradicts to the equation $u(D) = \dfrac{c(D)}{2}$. Therefore we have $c(D)=0$ and $D$ is a diagram of a trivial theta curve. 
\end{proof}

\noindent
\textit{Proof of Theorem\ref{theta}}\\
Let $f$ be a spatial embedding of $G$ such that $u(f)=\dfrac{c(f)}{2}$ and let $D$ be a minimal crossing diagram of $f(G)$. By Lemma\ \ref{mini} we have $u(D)=\dfrac{c(D)}{2}$. By Lemma \ref{dtheta} we see that $f$ is trivial.\qed

\begin{remark}\label{ge2} {\rm We can prove Theorem\ \ref{12c1} in the case of genus 2  by observing Lemma \ref{dhandcuff} and Lemma \ref{dtheta}. Let $D$ be a minimal crossing diagram of a non-trivial genus 2 handlebody-knot $H$. Then $D$ is also a diagram of a spatial embedding of a handcuff-graph or  a theta curve. 

In the case $D$ is a diagram of a spatial handcuff-graph, by Lemma \ref{dhandcuff} all crossings of $D$ are between two loops or $u(D) \leq \dfrac{c(D)-1}{2}$. In the case all crossings of $D$ are between two loops, by one IH-move on the edge that is not a loop we have a diagram $D'$ of $H$ such that $c(D')=c(D)=c(H)$ and $D'$ is also a diagram of a spatial theta curve (see Fig.\ \ref{wh}). By Lemma \ref{dtheta} we have $u(D') \leq \dfrac{c(D')-1}{2}$.  

In the case $D$ is a diagram of a spatial theta curve, by Lemma\ \ref{dtheta} we have $u(D) \leq \dfrac{c(D)-1}{2}$. In the both cases we have $u(H) \leq \dfrac{c(H)-1}{2}$.}
\end{remark}

\begin{figure}[H]
      \begin{tabular}{c}
      \begin{minipage}[b]{0.5\hsize}
        \begin{center}
          \includegraphics[width=0.35\linewidth]{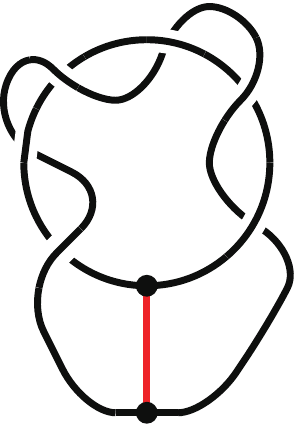}\vspace{0.5cm} \\
          $D$
        \end{center}
        
      \end{minipage}
      \begin{minipage}[b]{0.5\hsize}
        \begin{center}
          \includegraphics[width=0.35\linewidth]{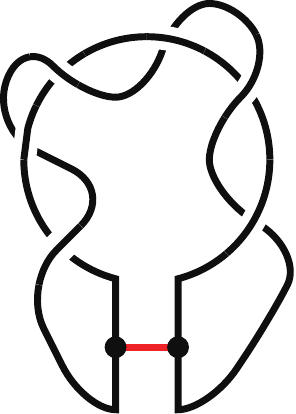}\vspace{0.5cm} \\
          $D'$
        \end{center}
      \end{minipage}
      \end{tabular}
      \caption{}
      \label{wh}
\end{figure}

\section{PROOFS OF THEOREM\ \ref{12c1} AND THEOREM\ \ref{12c2}}
In this section we prove Theorem \ref{12c1} and Theorem \ref{12c2}. In the following we give an inequality between unknotting number and crossing number by an observation of subdivided graph. 

Let $\tilde{f}:G \rightarrow \mathbb{R}^2$ be a regular projection of a graph $G$. Let $c_1,\ c_2,\ \cdots,\ c_k$ be crossing points of $\tilde{f}(G)$. A \textit{subdivided graph} of $G$ at $\{c_1,\ c_2,\ \cdots,\ c_k\}$ is a graph obtained by adding $2k$ vertices $v_1,\ v_1',\ v_2,\ v_2',\ \cdots ,v_k,\ v_k'$ to $G$ such that $\tilde{f}(v_i)=\tilde{f}(v_i')=c_i$ and $v_i$\ (resp.\ $v_i'$) is contained in the over-arc\ (resp.\ the under-arc) at $c_i$. Then we say that $v_i$\ (resp.\ $v_i')$ is an \textit{over-vertex}\ (resp.\ \textit{under-vertex}) at $c_i$. Let $T$ be a spanning tree of $G'$. For any two vertices $v$ and $u$ of $G'$, let $d_{T}(v,u)$ be the number of edges of the path in $T$ joining $v$ and $u$.

\begin{lemma}\label{31}
Let $D$ be a diagram of a nontrivial handlebody-knot $H$. Let $\tilde{f}:G \rightarrow \mathbb{R}^2$ be a regular projection of a connected trivalent graph $G$ where $D$ is obtained from $\tilde{f}(G)$. Let $c_1,\ c_2,\ \cdots,\ c_k$ be crossing points of $\tilde{f}(G)$. Let $G'$ be the subdivided graph of $G$ at $\{c_1,\ c_2,\ \cdots,\ c_k\}$. Let $v_1,\ v_1',\ v_2,\ v_2',\ \cdots ,v_k,\ v_k'$ be vertices of $G'$ such that $v_i$\ (resp.\ $v_i'$) is an over-vertex\ (resp.\ under-vertex) at $c_i\ (i=1,2,\cdots ,k)$. If there exists a vertex $v$ of $G'$ and a spanning tree $T$ of $G'$ such that $d_{T}(v,v_i)=d_{T}(v,v_i')$ for all $i \in \{1,2,\cdots,k\}$, then $u(D) \leq \dfrac{c(D)-k}{2}$. 
\end{lemma}
\begin{proof}
The proof is analogous to the proof of \cite[Proposition\ 3.2]{forbidden}. We fix a vertex $v$ of $G'$ and a spanning tree $T$ of $G'$ such that $d_{T}(v,v_i)=d_{T}(v,v_i')$ for all $i \in \{1,2,\cdots,k\}$\ (see Fig.\ \ref{311}). Let $h:G' \rightarrow \mathbb{R}$ be a continuous function with the following properties : 
\begin{figure}[H]
      \begin{tabular}{c}
      \begin{minipage}{0.5\hsize}
        \begin{center}
          \includegraphics[width=0.5\linewidth]{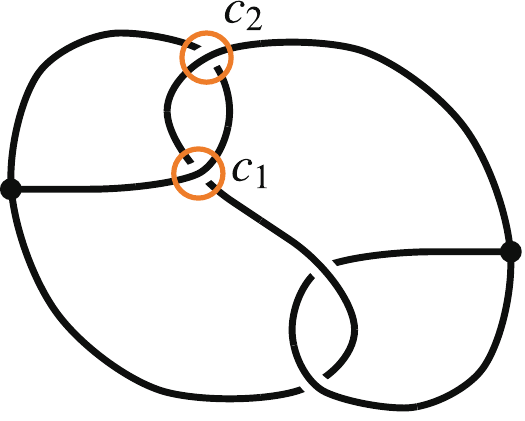}
        \end{center}
      \end{minipage}
      \begin{minipage}{0.5\hsize}
        \begin{center}
          \includegraphics[width=0.5\linewidth]{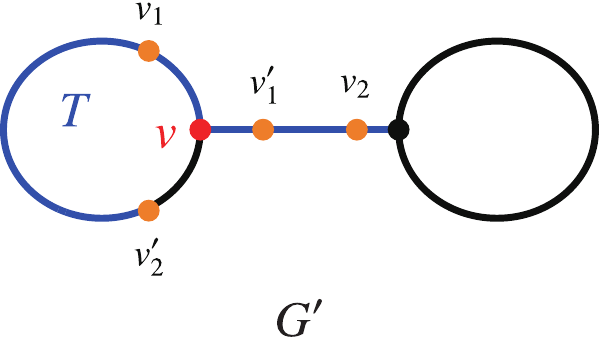}
        \end{center}
      \end{minipage}
      \end{tabular}
      \caption{}
      \label{311}
\end{figure}
\noindent
$(1)\,$For each vertex $u$ of $T$, $h|_T(u)=-d_{T}(v,u)$.\\
$(2)\,h|_e$ is injective for each edge $e$ of $T$.\\
$(3)\,$Each edge of $G'-T$ has exactly one minimum point of $h$. \\
We can deform $h$ slightly so that $h(v_i)>h(v_i')\ (i=1,2,\cdots,k)$ since $d_{T}(v,v_i)=d_{T}(v,v_i')\ (i=1,2,\cdots,k)$. Then we give over/under informations to $\tilde{f}$ to produce a spatial embedding $f : G \rightarrow \mathbb{R}^3 = \mathbb{R}^2 \times \mathbb{R}$ such that $p_1 \circ f=\tilde{f}$ and $p_2 \circ f =h$, where $p_1$\ (respectively $p_2)$ denotes the projection of $\mathbb{R}^3$ to the first factor (respectively to the second factor) of $\mathbb{R}^2 \times \mathbb{R}$.

Let $D'$ be the diagram of $f(G)$ where $D'$ is obtained from $\tilde{f}(G)$. We note that $D$ and $D'$ are deformed  into each other by changing over/under informations of crossing points without changing over/under informations of $c_1,\ c_2,\ \cdots,\ c_k$\ (see for example Fig.\ \ref{312}). Since we obtain a bouquet as in Fig.\ \ref{313} which is trivial by contracting spatial edges of $f(T)$, $D'$ is a diagram of a trivial handlebody-knot. 

\begin{figure}[H]
      \begin{tabular}{c}
      \begin{minipage}{0.5\hsize}
        \begin{center}
          \includegraphics[width=0.6\linewidth]{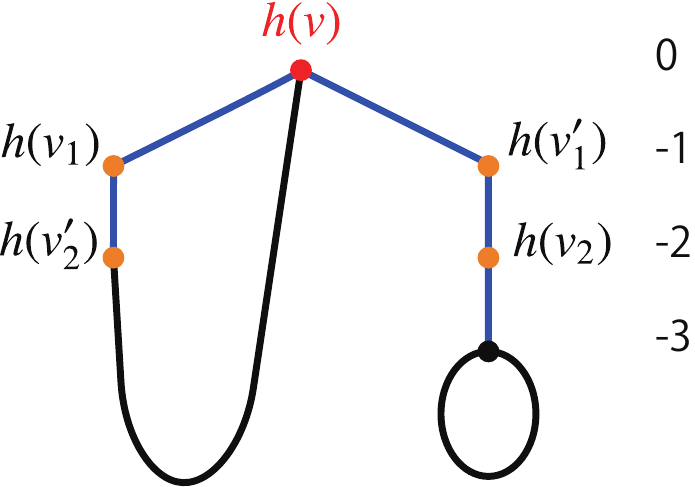}
        \end{center}
      \end{minipage}
      \begin{minipage}{0.5\hsize}
        \begin{center}
          \includegraphics[width=0.5\linewidth]{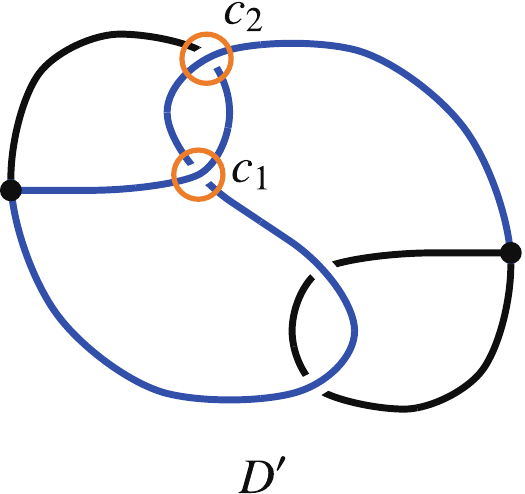}
        \end{center}
      \end{minipage}
      \end{tabular}
      \caption{}
      \label{312}
\end{figure}

\begin{figure}[H]
      \begin{tabular}{c}
      \begin{minipage}{0.5\hsize}
        \begin{center}
          \includegraphics[width=0.4\linewidth]{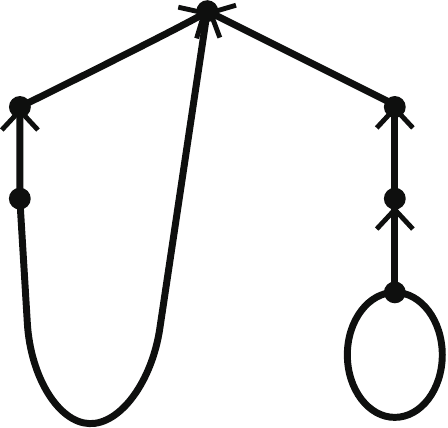}
        \end{center}
      \end{minipage}
      \begin{minipage}{0.5\hsize}
        \begin{center}
          \includegraphics[width=0.4\linewidth]{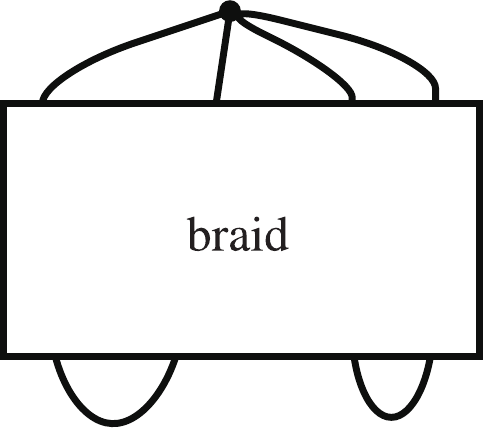}
        \end{center}
      \end{minipage}
      \end{tabular}
      \caption{}
      \label{313}
\end{figure}

Let $D''$ be the diagram that is obtained from $D'$ by changing over/under informations of all crossing points with the exception of  $c_1,\ c_2,\ \cdots,\ c_k$. Let $h':G' \rightarrow \mathbb{R}$ be a continuous function such that $h'=-h$. We can deform $h'$ slightly so that $h'(v_i)>h'(v_i')\ (i=1,2,\cdots,k)$\  (see for example Fig.\ \ref{314}). Then $D''$ is the diagram of a spatial embedding of $f' : G \rightarrow \mathbb{R}^3 = \mathbb{R}^2 \times \mathbb{R}$ such that $p_1 \circ f'=\tilde{f}$ and $p_2 \circ f'=h'$. Thus we obtain a trivial bouquet by contracting spatial edges of $f'(T)$ and $D''$ is a diagram of a trivial handlebody-knot. 

\begin{figure}[H]
      \begin{tabular}{c}
      \begin{minipage}{0.5\hsize}
        \begin{center}
          \includegraphics[width=0.5\linewidth]{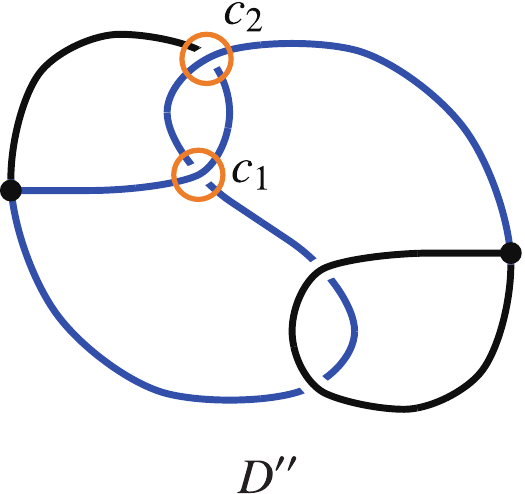}
        \end{center}
      \end{minipage}
      \begin{minipage}{0.5\hsize}
        \begin{center}
          \includegraphics[width=0.6\linewidth]{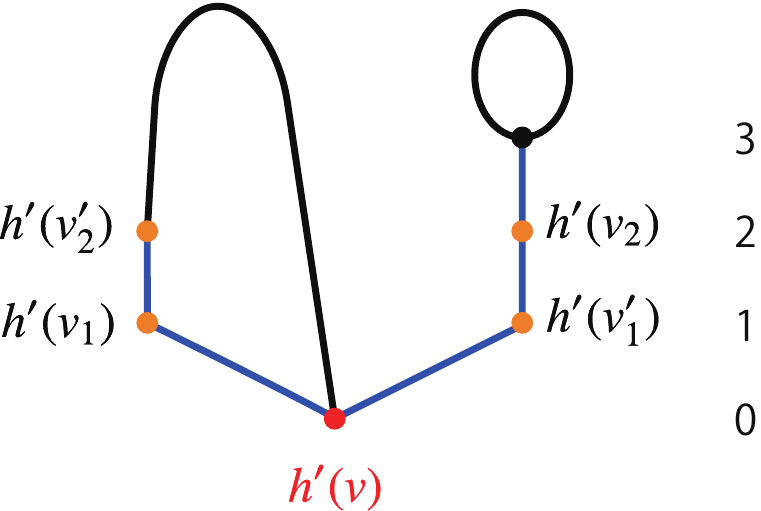}
        \end{center}
      \end{minipage}
      \end{tabular}
      \caption{}
      \label{314}
\end{figure}

Let $A$ be a subset of $C(D)$ such that changing all crossings in $A$ turns $D$ to $D'$. We note that changing all crossings in $(\ C(D)-\{c_1,\ c_2,\ \cdots,\ c_k\}\ )-A$ turns $D$ to $D''$. Therefore we have
$$u(D) \leq \min \{|A|,|(\ C(D)-\{c_1,\ c_2,\ \cdots,\ c_k\}\ )-A|\} \leq \dfrac{c(D)-k}{2}$$ 
\end{proof}
\noindent
\textit{Proof of Theorem \ref{12c1}}\\
Let $D$ be a minimal crossing diagram of a nontrivial handlebody-knot $H$. Let $\tilde{f}:G \rightarrow \mathbb{R}^2$ be a regular projection of a connected trivalent graph $G$ where $D$ is obtained from $\tilde{f}(G)$. Let $c_1$ be a crossing point of $\tilde{f}(G)$. Let $G'$ be the subdivided graph of $G$ at $\{c_1\}$. Let $v_1$\ (resp.\ $v_1'$) be a vertex of $G'$ such that $v_1\ (resp.\ v_1')$ is over-vertex (resp.\ under-vertex) at $c_1$. Let $T$ be a spanning tree of $G'$ containing $v_1$ and $v_1'$. \\
\begin{figure}[H]
      \begin{tabular}{c}
      \begin{minipage}{0.45\hsize}
        \begin{center}
          \includegraphics[width=0.5\linewidth]{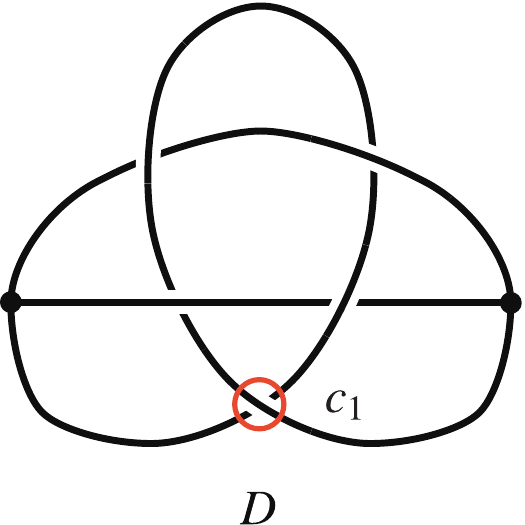}
        \end{center}
      \end{minipage}
      \begin{minipage}{0.45\hsize}
        \begin{center}
          \includegraphics[width=0.5\linewidth]{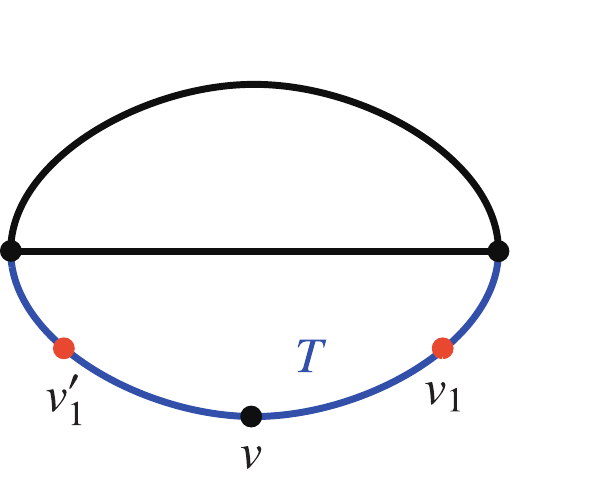}
        \end{center}
      \end{minipage}
      \end{tabular}
      \caption{}
      \label{141}
\end{figure}
By subdividing $T$ if necessary, we can choose a vertex $v$ of $T$ such that $d_{T}(v,v_1)=d_{T}(v,v_1')$ since there exists a path in $T$ joining $v_1$ and $v_1'$\ (see for example Fig.\ \ref{141}). Then by Lemma \ref{31} we have $u(D) \leq \dfrac{c(D)-1}{2}$. Since $u(H) \leq u(D)$ and $c(D)=c(H)$ we have $u(H) \leq \dfrac{c(H)-1}{2}$. \qed

\begin{lemma}\label{d12c2}
Let $D$ be a minimal crossing diagram of a nontrivial handlebody-knot $H$ that satisfies the equality $u(D)=\dfrac{c(D)-1}{2}$. Let $\gamma$ be a cycle of $D$ that has at least one crossing of $D$. Then the following $(1)$ and $(2)$ holds.\\
$(1)\,$All crossings of $D$ are self-crossings of $\gamma$.\\
$(2)\,$There exists an odd number $p \neq \pm1$ such that $\gamma$ is a reduced alternating diagram of a $(2,\ p)$-torus knot.
\end{lemma}
\begin{proof}
Suppose that $\gamma$ has just one crossing. Suppose that $\gamma$ itself is a simple closed curve on $\mathbb{R}^2$. Then we have a diagram $D'$ of $H$ as illustrated in Fig.\ \ref{cl1} such that $c(D')=c(D)-1$. Suppose that $\gamma$ is not a simple closed curve on $\mathbb{R}^2$ and $\gamma$ has exactly one crossing of $D$. Then by a similar deformation we have a diagram $D'$ of $H$ with $c(D')=c(D)-1$. This contradicts that $D$ is a minimal crossing diagram. Thus $\gamma$ has at least two crossings of $D$. \\
\begin{figure}[h]
      \begin{tabular}{c}
      \begin{minipage}{0.3\hsize}
        \begin{center}
          \includegraphics[width=0.8\linewidth]{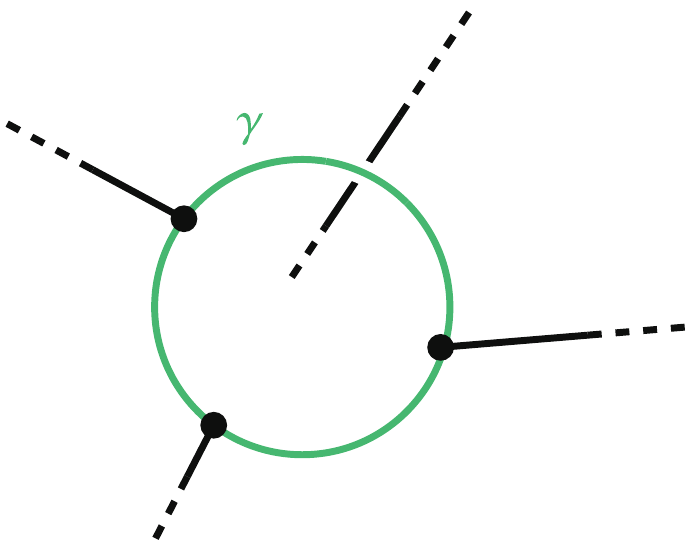}\\
          $D$
        \end{center}
      \end{minipage}

      \begin{minipage}{0.3\hsize}
        \begin{center}
          \includegraphics[width=0.8\linewidth]{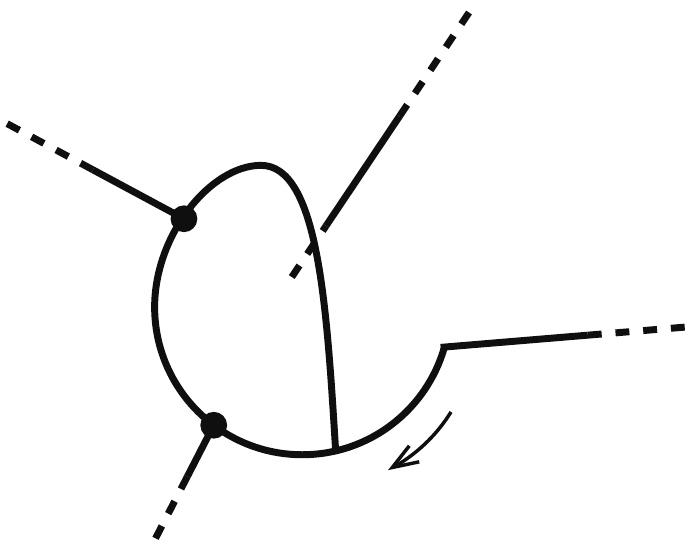}\\
          \ \ 
        \end{center}
      \end{minipage}

      \begin{minipage}{0.3\hsize}
        \begin{center}
          \includegraphics[width=0.8\linewidth]{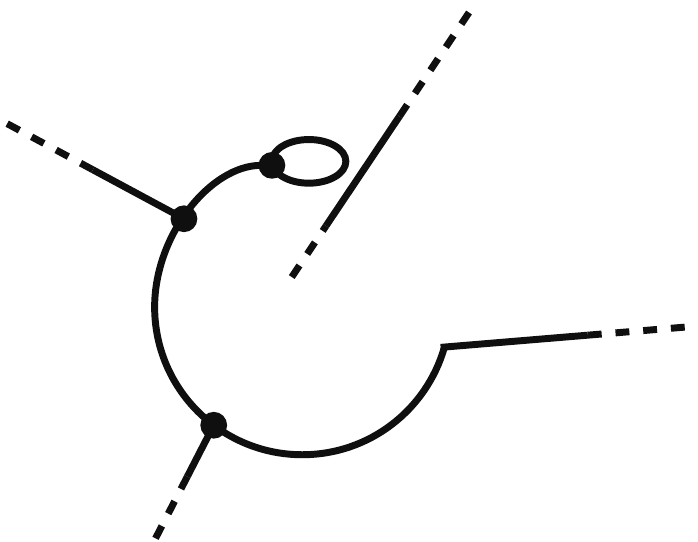}\\
          $D'$
        \end{center}
      \end{minipage}
      \end{tabular}
      \caption{}
      \label{cl1}
\end{figure}\\
Let $\tilde{f}$ be a regular projection of a trivalent graph $G$ where $D$ is obtained from $\tilde{f}(G)$. First, we show that if $(1)$ does not hold, then $u(D)\leq \dfrac{c(D)-2}{2}$. We will show this claim step by step as follows.\\
\ \\
\textbf{Subclaim 1.}\ \textit{If one of the crossings on $\gamma$, say $c_1$, is a crossing  between $\gamma$ and $D-\gamma$ and another crossing on $\gamma$, say $c_2$, is a self-crossing of $\gamma$, then $u(D) \leq \dfrac{c(D)-2}{2}$.}
\begin{proof}
Let $G'$ be the subdivided graph of $G$ at $\{c_1,\ c_2\}$. Let $v_i$\ (resp.\ $v_i')$ be the over-vertex\ (resp.\ under-vertex) at $c_i$\ $(i=1,\ 2)$. Then $G'$ is the graph as illustrated in Fig.\ \ref{slmu}\ $(a)$ or $(b)$. 


\begin{figure}[H]
      \begin{tabular}{c}
      \begin{minipage}{0.5\hsize}
        \begin{center}
          \includegraphics[height=0.4\linewidth]{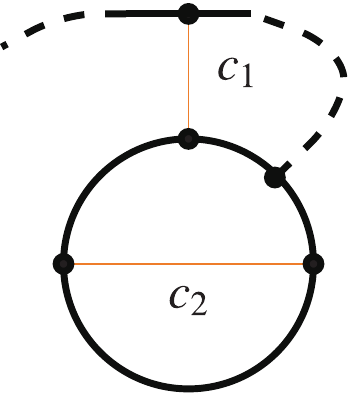}
        \end{center}
        \subcaption{}
      \end{minipage}
      \begin{minipage}{0.5\hsize}
        \begin{center}
          \includegraphics[height=0.4\linewidth]{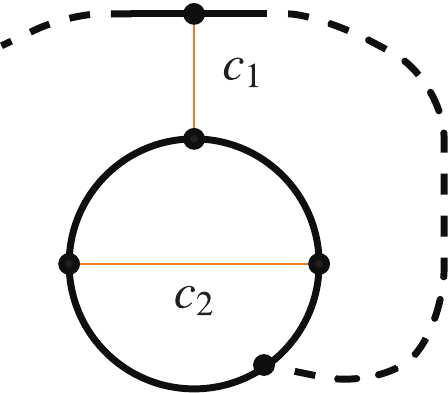}
        \end{center}
        \subcaption{}
      \end{minipage}
      \end{tabular}
      \caption{}
      \label{slmu}
\end{figure}

\begin{figure}[H]
        \begin{center}
          \includegraphics[width=0.2\linewidth]{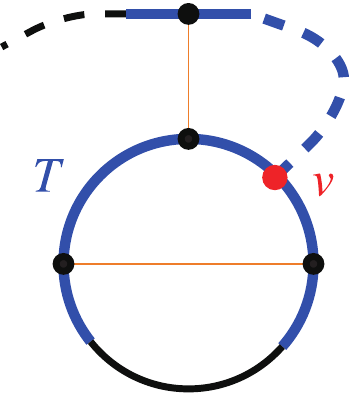}
        \end{center}     
      \caption{}
      \label{slmua}
\end{figure}

By subdividing if necessary, we can choose a spanning tree $T$ of $G'$ and a vertex $v$ of $T$ such that $d_T(v,v_i)=d_T(v,v_i')\ (i=1,2)$. A choice of $T$ and  $v$ for the case of Fig.\ \ref{slmu}\ $(a)$ is illustrated in Fig.\ \ref{slmua}. By Lemma \ref{31} we have $u(D)\leq \dfrac{c(D)-2}{2}$. 
\end{proof}
\noindent
\textbf{Subclaim 2.}\ \textit{If two of crossings on $\gamma$, say $c_1$ and $c_2$, are  crossings  between $\gamma$ and $D-\gamma$ then $u(D) \leq \dfrac{c(D)-2}{2}$.} 
\begin{proof}
Let $G'$ be the subdivided graph of $G$ at $\{c_1,\ c_2\}$. Let $v_i$\ (resp.\ $v_i')$ be the over-vertex\ (resp.\ under-vertex) at $c_i$\ $(i=1,\ 2)$. Then $G'$ is one of the graphs as illustrated in Fig.\ \ref{mumu}. 

\begin{figure}[H]
      \begin{tabular}{c}
      \begin{minipage}{0.24\hsize}
        \begin{center}
          \includegraphics[height=0.8\linewidth]{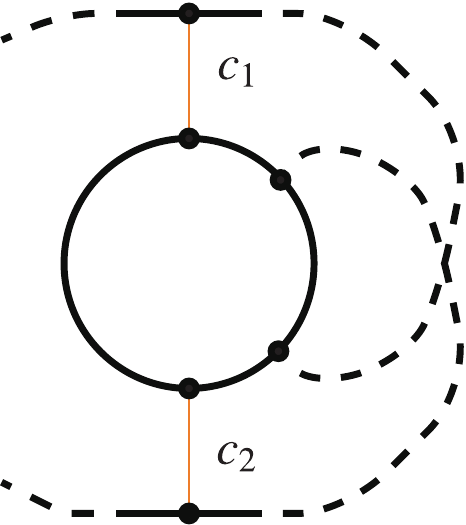}
        \end{center}
        \subcaption{}
      \end{minipage}
      \begin{minipage}{0.24\hsize}
        \begin{center}
          \includegraphics[height=0.8\linewidth]{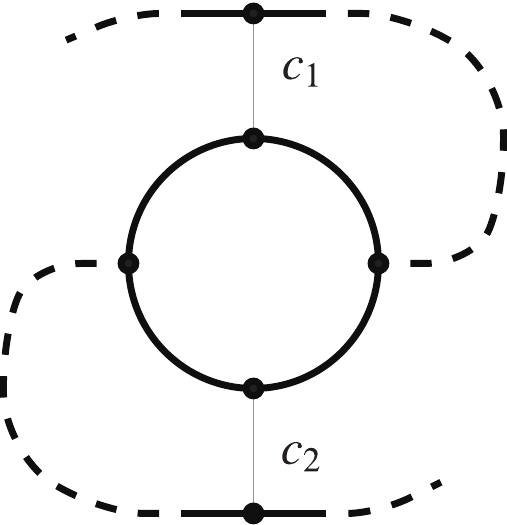}
        \end{center}
        \subcaption{}
        \end{minipage}
        \begin{minipage}{0.24\hsize}
        \begin{center}
          \includegraphics[height=0.8\linewidth]{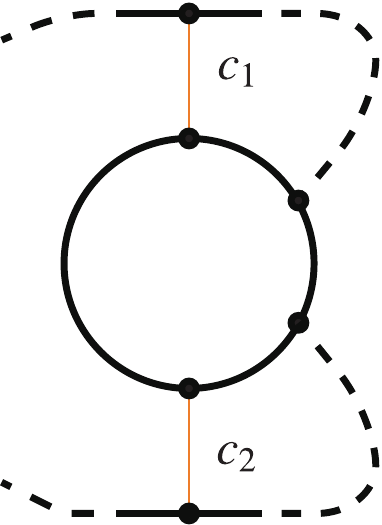}
        \end{center}
        \subcaption{}
      \end{minipage}
      \begin{minipage}{0.24\hsize}
        \begin{center}
          \includegraphics[height=0.8\linewidth]{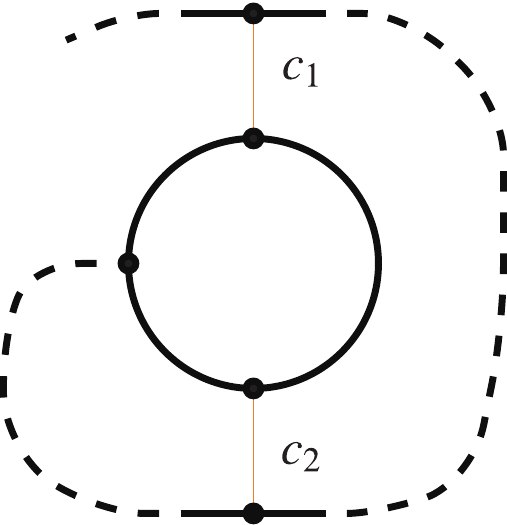}
        \end{center}
        \subcaption{}
      \end{minipage}
      
      \end{tabular}
      
      \caption{}
      \label{mumu}
\end{figure}

\begin{figure}[H]
        \begin{center}
          \includegraphics[width=0.25\linewidth]{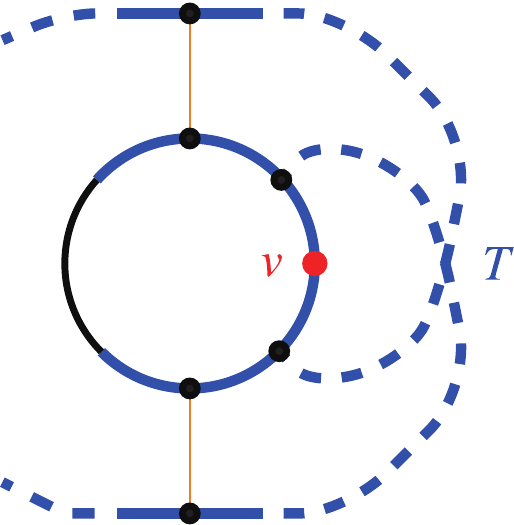}
        \end{center}     
      \caption{}
      \label{mumua}
\end{figure}

By subdividing if necessary, we can choose a spanning tree $T$ of $G'$ containing $v_1,\,v_1',\,v_2,\,v_2'$ and a vertex $v$ of $T$ such that $d_T(v,v_i)=d_T(v,v_i')\ (i=1,2)$. A choice of $T$ and $v$ for the case of Fig.\ \ref{mumu}\ $(a)$ is illustrated in Fig.\ \ref{mumua}. By Lemma \ref{31} we have $u(D)\leq \dfrac{c(D)-2}{2}$. 
\end{proof}
\noindent
\textbf{Subclaim 3.}\ \textit{If there exists a self-crossing of $D-\gamma$, say $c_1$,  then $u(D) \leq \dfrac{c(D)-2}{2}$.} 
\begin{proof}
By Subclaim 2 we may assume that $\gamma$ has a self-crossing, say $c_2$. Let $G'$ be the subdivided graph of $G$ at $\{c_1,\ c_2\}$. Let $v_i$\ (resp.\ $v_i')$ be the over-vertex\ (resp.\ under-vertex) at $c_i$\ $(i=1,\ 2)$. Then $G'$ is one of the graphs as illustrated in Fig.\ \ref{sese}. 

\begin{figure}[H]
      \begin{tabular}{c}
      \begin{minipage}{0.24\hsize}
        \begin{center}
          \includegraphics[height=0.42\linewidth]{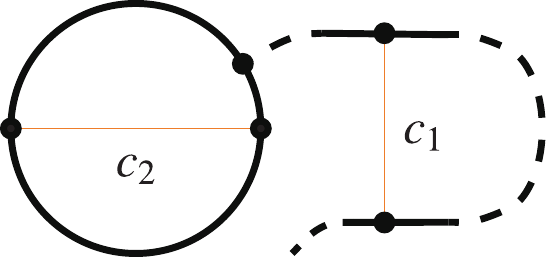}
        \end{center}
        \subcaption{}
      \end{minipage}
      \begin{minipage}{0.24\hsize}
        \begin{center}
          \includegraphics[height=0.42\linewidth]{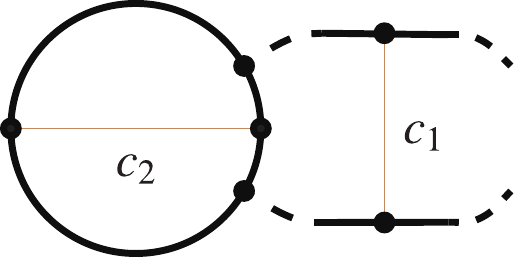}
        \end{center}
        \subcaption{}
        \end{minipage}
        \begin{minipage}{0.24\hsize}
        \begin{center}
          \includegraphics[height=0.42\linewidth]{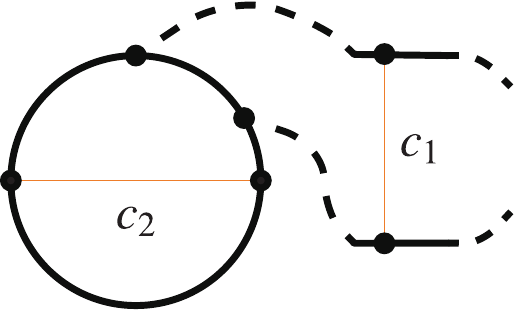}
        \end{center}
        \subcaption{}
      \end{minipage}
      \begin{minipage}{0.24\hsize}
        \begin{center}
          \includegraphics[height=0.42\linewidth]{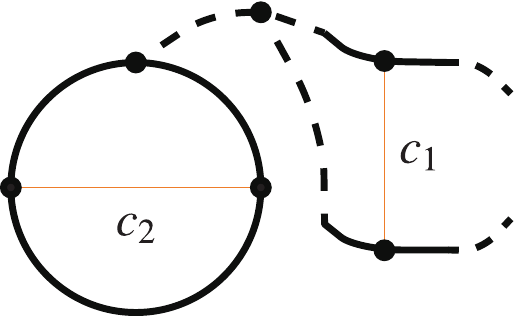}
        \end{center}
        \subcaption{}
      \end{minipage}
      
      \end{tabular}
      \caption{}
      \label{sese}
\end{figure}

\begin{figure}[H]
        \begin{center}
          \includegraphics[width=0.25\linewidth]{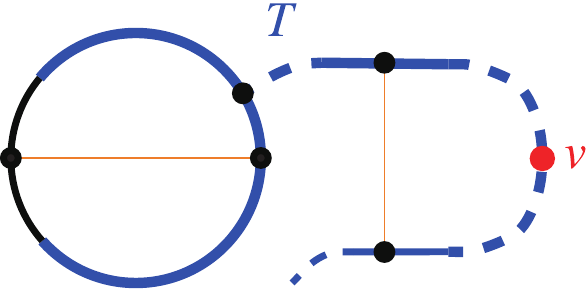}
        \end{center}     
      \caption{}
      \label{sesea}
\end{figure}

By subdividing if necessary, we can choose a spanning tree $T$ of $G'$ containing $v_1,\,v_1',\,v_2,\,v_2'$ and a vertex $v$ of $T$ such that $d_T(v,v_i)=d_T(v,v_i')\ (i=1,2)$.  A choice of $T$ and $v$ for the case Fig.\ \ref{sese}\ $(a)$ is illustrated in Fig.\ \ref{sesea}.  By Lemma \ref{31} we have $u(D)\leq \dfrac{c(D)-2}{2}$. 
\end{proof}
 
From the above we see that $\gamma$ satisfies $(1)$. \\
\ \\
\textbf{Subclaim 4.}\ \textit{If $\gamma$ is not obtained from a standard projection of a $(2,\ p)$-torus knot for any odd number $p>1$, then $u(D) \leq \dfrac{c(D)-2}{2}$.}

\begin{proof}

Let $G'$ be the subdivided graph of $G$ at $C(D)$ and let $\Gamma$ be a cycle of $G'$ such that $\gamma$ is obtained from $\tilde{f}(\Gamma)$. Note that if $\tilde{f}(\Gamma)$ is a standard projection of a $(2,\ p)$-torus knot for some odd number $p\neq \pm 1$ as the case $p=5$ is illustrated in the left of Fig.\ \ref{pro25}, then any pair of crossings on $\gamma$ are not parallel. Namely $\Gamma$ is as illustrated in the right of Fig.\ \ref{pro25}. It follows from \cite[Theorem\ 1]{proje} that the converse is also true (see also \cite[Proof of Theorem\ 1.11]{pseudo}).

\begin{figure}[h]
      \begin{tabular}{c}
      
      \begin{minipage}[b]{0.5\hsize}
        \begin{center}
          \includegraphics[width=0.4\linewidth]{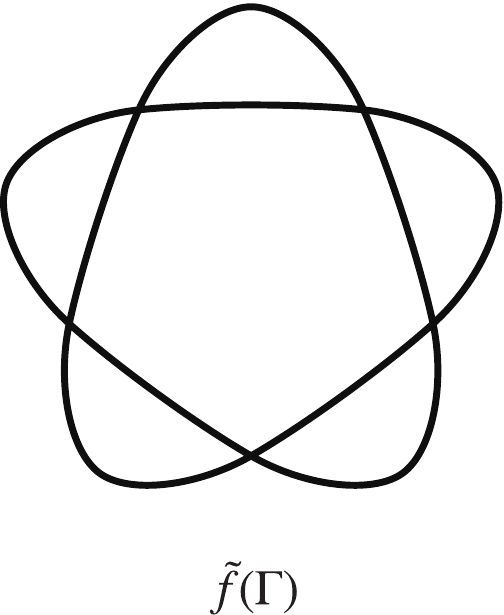}
        \end{center}
      \end{minipage}
      
      \begin{minipage}[b]{0.5\hsize}
        \begin{center}
          \includegraphics[width=0.4\linewidth]{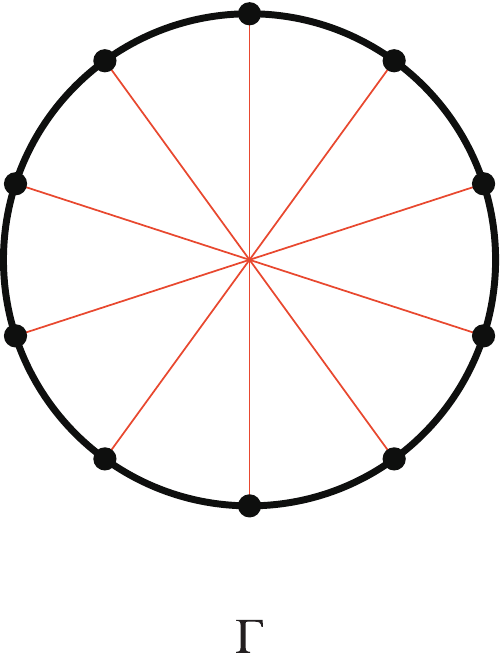}
        \end{center}
      \end{minipage}
      
      \end{tabular}
      \caption{}
      \label{pro25}
\end{figure}

Therefore there are two crossings $c_1$, $c_2$ of $\gamma$ such that $c_1$ and $c_2$ are parallel, namely $\Gamma$ is illustrated as the left of Fig.\ \ref{para}. Let $v_i$\ (resp.\ $v_i')$ be the over-vertex\ (resp.\ under-vertex) at $c_i$\ $(i=1,\ 2)$. By subdividing if necessary, we can choose a spanning tree $T$ of $G'$ and a vertex $v$ of $T$ such that $d_T(v,v_i)=d_T(v,v_i')\ (i=1,2)$\ (see the right of Fig.\ \ref{para}). By Lemma \ref{31} we have $u(D)\leq \dfrac{c(D)-2}{2}$. 
\begin{figure}[h]
      \begin{tabular}{c}
      \begin{minipage}[t]{0.45\hsize}
        \begin{center}
          \includegraphics[width=0.4\linewidth]{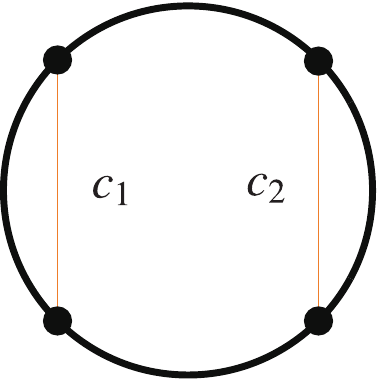}
        \end{center}
      \end{minipage}

      \begin{minipage}[t]{0.45\hsize}
        \begin{center}
          \includegraphics[width=0.4\linewidth]{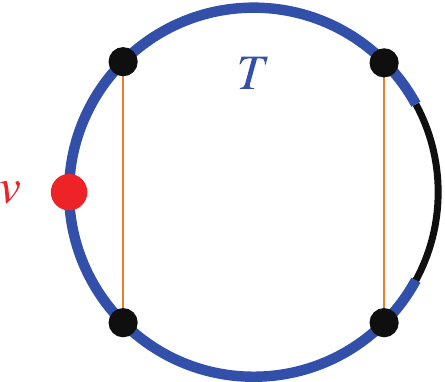}
        \end{center}
      \end{minipage}
      \end{tabular}
      \caption{}
      \label{para}
\end{figure}

\end{proof}

Finally, we show that if $u(D)=\dfrac{c(D)-1}{2}$, then $u(\gamma)=u(D)$ and $\gamma$ is a reduced alternating diagram of a $(2,\ p)-$ torus knot. 

Let $\Gamma$ be a cycle of $G$ such that $\gamma$ is obtained from $\tilde{f}(\Gamma)$. From the above $\tilde{f}(\Gamma)$ is a standard projection of a $(2,\ p)-$torus knot as the case $p=5$ is illustrated in the left of Fig.\ \ref{pro25} and $c(\gamma)=c(D)=p$. If we can join two components of $\tilde{f}(\Gamma) \backslash C(\tilde{f}(G))$ by a path $P$ of $\tilde{f}(G)$ as illustrated in the left of Fig.\ \ref{pro2}, then there exists a cycle $\gamma'$ of $D$ that has a crossing between $\gamma$ and $D-\gamma$ as illustrated in the right of Fig.\ \ref{pro2}. This is contradict to Lemma\ \ref{d12c2}\ $(1)$. 

\begin{figure}[h]
      \begin{tabular}{c}
      \begin{minipage}[t]{0.45\hsize}
        \begin{center}
          \includegraphics[width=0.5\linewidth]{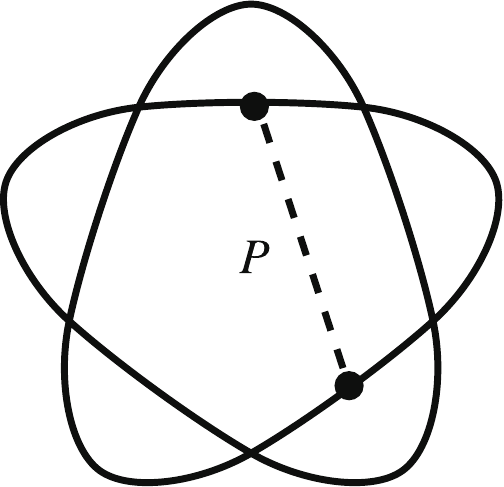}
        \end{center}
      \end{minipage}

      \begin{minipage}[t]{0.45\hsize}
        \begin{center}
          \includegraphics[width=0.5\linewidth]{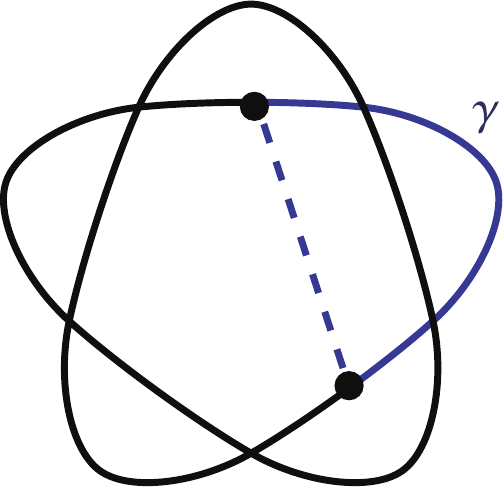}
        \end{center}
      \end{minipage}
      \end{tabular}
      \caption{}
      \label{pro2}
\end{figure}

Therefore we may assume that $\tilde{f}(G)$ has no paths as illustrated in the left of Fig.\ \ref{pro2}, namely $\tilde{f}(G)$ is a projection as illustrated in Fig.\ \ref{pro3}. Then by changing over/under informations at $u(\gamma)$ crossings on $D$ we can obtain a diagram of a trivial handlebody-knot. Since $c(\gamma)=c(D),\ u(D)=\dfrac{c(D)-1}{2}$ and $u(D) \leq u(\gamma)$, we have $u(\gamma)=\dfrac{c(\gamma)-1}{2}$. By Theorem \ref{12c2k}, $\gamma$ is a reduced alternating diagram of a $(2,\ p)$-torus knot for some odd number $p \neq \pm 1$. 

\begin{figure}[h]
\centering
  \includegraphics[width=0.25\linewidth]{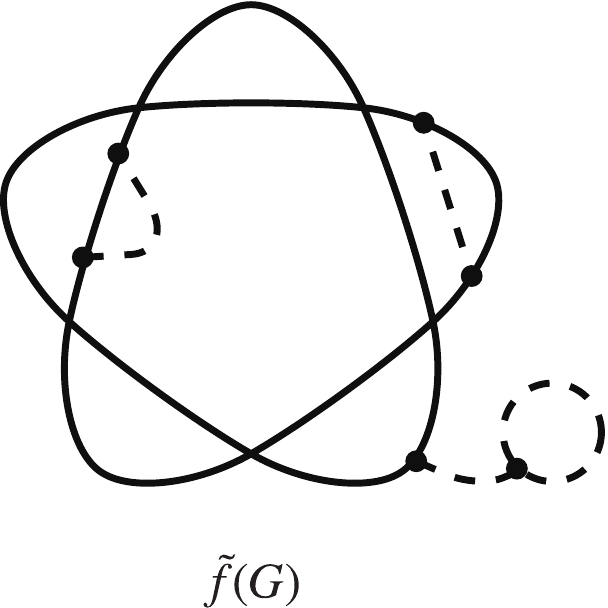}
  \caption{}
      \label{pro3}
\end{figure}

\end{proof}

\noindent
\textit{Proof of Thm \ref{12c2}}\\
Let $H$ be a nontrivial handlebody-knot that satisfies the equality $u(H) = \dfrac{c(H) -1}{2}$. Let $D$ be a minimal crossing diagram of $H$. Since $u(H) \leq u(D)$ and $\dfrac{c(D)-1}{2} =\dfrac{c(H)-1}{2}$ we have $u(D) \geq \dfrac{c(D)-1}{2}$. Thus by the proof of Theorem\ \ref{12c1} we have $u(D) = \dfrac{c(D)-1}{2}$. Then by Lemma\ \ref{d12c2} we see that $H$ is a handlebody-knot represented by one of diagrams illustrated in Fig.\ \ref{d2n1}. We note that the unknotting number of handlebody-knot represented by $D_{2n-1}(n \ne 0,1)$ are one by Proposition \ref{2br}. Therefore $H$ is a handlebody-knot represented by $D_3$ or $D_{-3}$ as desired. \qed

\section*{Acknowledgements}
The author would like to thank Professor Tomo Murao for his helpful comments. He is particularly grateful to Professor Kouki Taniyama for invaluable advice and
his suggestions.

\bibliographystyle{myplain2}
\bibliography{citations.bib}

\begin{thebibliography}{10}

\bibitem{uncr}
Y.~Akimoto and K.~Taniyama,
\newblock Unknotting numbers and crossing numbers of spatial embeddings of a
  planar graph,
\newblock {\em J. Knot Theory Ramifications}, 29(14):2050095, 13, 2020.

\bibitem{ob}
B.~Dorothy and O.~Danielle,
\newblock Unknotting numbers for prime $\theta$-curves up to seven crossings,
\newblock preprint. (arXiv:1710.05237).

\bibitem{proje}
C.~H. Dowker and M.~B. Thistlethwaite,
\newblock Classification of knot projections,
\newblock {\em Topology Appl.}, 16(1):19--31, 1983.

\bibitem{pseudo}
R.~Hanaki,
\newblock Pseudo diagrams of knots, links and spatial graphs,
\newblock {\em Osaka J. Math.}, 47(3):863--883, 2010.

\bibitem{handle}
A.~Ishii,
\newblock Moves and invariants for knotted handlebodies,
\newblock {\em Algebr. Geom. Topol.}, 8(3):1403--1418, 2008.

\bibitem{unqc}
M.~Iwakiri,
\newblock Unknotting numbers for handlebody-knots and {A}lexander quandle
  colorings,
\newblock {\em J. Knot Theory Ramifications}, 24(14):1550059, 13, 2015.

\bibitem{tunnel}
T.~Kobayashi,
\newblock A criterion for detecting inequivalent tunnels for a knot,
\newblock {\em Math. Proc. Cambridge Philos. Soc.}, 107(3):483--491, 1990.

\bibitem{forbidden}
R.~Nikkuni, M.~Ozawa, K.~Taniyama, and Y.~Tsutsumi,
\newblock Newly found forbidden graphs for trivializability,
\newblock {\em J. Knot Theory Ramifications}, 14(4):523--538, 2005.

\bibitem{class}
I.~Sugiura and S.~Suzuki,
\newblock On a class of trivializable graphs,
\newblock {\em Sci. Math.}, 3(2):193--200, 2000.

\bibitem{tamura}
N.~Tamura,
\newblock On an extension of trivializable graphs,
\newblock {\em J. Knot Theory Ramifications}, 13(2):211--218, 2004.

\bibitem{knotted}
K.~Taniyama,
\newblock Knotted projections of planar graphs,
\newblock {\em Proc. Amer. Math. Soc.}, 123(11):3575--3579, 1995.

\bibitem{unbound}
K.~Taniyama,
\newblock Unknotting numbers of diagrams of a given nontrivial knot are
  unbounded,
\newblock {\em J. Knot Theory Ramifications}, 18(8):1049--1063, 2009.

\end{thebibliography}

\end{document}